\UseRawInputEncoding
\documentclass[12pt]{amsart}
\usepackage{epsfig, color, amsmath, esint, hyperref, mathrsfs, xcolor, bm, enumitem, mathtools, comment, amsfonts, amssymb}

\usepackage[backend=biber,style=alphabetic]{biblatex}
\hypersetup{colorlinks}
\hypersetup{citecolor=blue}
\hypersetup{urlcolor=blue}
\makeatother
\theoremstyle{definition}
\def\fnum{equation} 
\newtheorem{Thm}[\fnum]{Theorem}
\newtheorem{Cor}[\fnum]{Corollary}

\newtheorem{Lem}[\fnum]{Lemma}

\newtheorem{Exa}[\fnum]{Example}
\newtheorem{Rem}[\fnum]{Remark}
\newtheorem{Pro}[\fnum]{Proposition}
\newtheorem{Def}[\fnum]{Definition}
\newtheorem{Not}[\fnum]{Notation}
\newtheorem{mainthm}{Theorem}

\numberwithin{equation}{section}

\newcommand{\inte}{{\text {int}}}

\newcommand{\Id}{\text{Id}}
\newcommand{\iso}{\text{Iso}}
\newcommand{\vv}{\mathrm{v}}
\newcommand{\ww}{\mathrm{w}}
\newcommand{\xx}{\mathrm{x}}
\newcommand{\yy}{\mathrm{y}}
\newcommand{\zz}{\mathrm{z}}
\newcommand{\ad}{\mathrm{ad}}
\newcommand{\Ad}{\mathrm{Ad}}
\newcommand{\vol}{\text{vol}}

\newcommand{\ric}{\text{Ric}}

\title{Convergence of symmetries:   
nilpotency,  dimension, and perfectness}
\author{Sergio Zamora}
\address{\parbox{\linewidth}{Sergio Zamora\\ Beijing International Center for Mathematical Research\\  zamorabs@bicmr.pku.edu.cn}}

\begin{document}

\maketitle

\begin{abstract}
Let  $(X_i,p_i)$  be a sequence of pointed $n$-dimensional Riemannian manifolds with a uniform lower Ricci curvature bound, and  $G_i \leq \iso (X_i)$ a sequence of closed groups of isometries.

We show that if the triples $(X_i, G_i, p_i)$ converge in the equivariant Gromov--Hausdorff sense to a triple $(X,G,p)$, then $\mathfrak{g}$, the Lie algebra of $G$, admits an ideal $\mathfrak{h} \trianglelefteq \mathfrak{g}$ with $\dim (\mathfrak{h}) \leq  \limsup_i \dim (G_i)$ and $\mathfrak{g}/ \mathfrak{h}$ nilpotent. Moreover, if the sequence $(X_i, p_i ) $ is non-collapsing, we show that $\mathfrak{h} $ can be taken of dimension $\limsup_i \dim (G_i)$.

\end{abstract}

{
\hypersetup{linkcolor=black}
\tableofcontents
}

\section{Introduction}

Equivariant Gromov--Hausdorff convergence was introduced by Fukaya in \cite{fukaya} and has proven to be a  successful tool for studying spaces  with lower curvature bounds (see for example \cite{fukaya-yamaguchi, kapovitch-petrunin-tuschmann, kapovitch-wilking, rong}). In this paper, we specialize in the following setting. 

Let  $(X_i,p_i)$  be a sequence of pointed complete $n$-dimensional Riemannian manifolds with $\ric (X_i) \geq - (n-1)$, and  $G_i \leq \iso (X_i)$ a sequence of closed groups of isometries. After passing to a subsequence, we can assume the triples $(X_i, G_i, p_i)$ converge in the equivariant Gromov--Hausdorff sense to a triple $(X, G, p)$.  Hereafter we will denote this by 
\[   (X_i, G_i , p_i) \xrightarrow{eGH} (X,G,p) .     \]

The analytic, geometric, and topological properties of $X$ have been studied extensively (see for example  \cite{colding, cheeger-colding-i, cheeger-colding-ii, cheeger-colding-iii, sormani-wei-i, sormani-wei-ii, colding-naber, cheeger-naber, pan-rong, pan-wei, pan-wang, wang}).  In order to further investigate such properties, and to relate them to the corresponding geometric and topological features of the spaces $X_i$, an important problem is to understand the relationship between the groups $G_i$ and $G$. Let us recall some known results:
\begin{enumerate}
    \item Colding and Naber \cite{colding-naber} showed that the isometry group $\iso (X)$ is a Lie group, hence so is  $G$. This result was proven earlier in \cite{fukaya-yamaguchi-lie} assuming the manifolds $X_i$ have a uniform lower sectional curvature bound. \label{known-i}
    \item When the orbits $ G_i \cdot p_i $ are uniformly bounded, Grove, Karcher, and Ruh \cite{grove-karcher-ruh} showed that the maps $\phi _i : G_i \to G$ demonstrating the convergence can be taken to be continuous group homomorphisms. This was done assuming $G$ is a Lie group, which is always true due to \eqref{known-i}. See also \cite{mazur-rong-wang, harvey, alattar} for related results and applications in differential geometry.\label{known-ii}
    \item When the sequence $(X_i, p_i)$ is non-collapsing, Pan and Rong \cite{pan-rong} showed that the sequence $\iso (X_i)$ has the no small subgroup property. That is, if $ \vol (B_1(p_i)) \geq v$  for some $v >0$, the orbits $ G_i \cdot p_i $ are uniformly bounded,  and $G$ is trivial, then $G_i$ is also trivial for  $i$ large enough.\label{known-pan-rong}
    \item When the groups $G_i$ are discrete, Breuillard, Green, and Tao \cite{breuillard-green-tao} gave a very explicit description of the structure of the groups $G_i$ near the identity. One consequence of this structure is that the identity component of $G$ is nilpotent. See also \cite{turing, kapovitch-wilking, wang-nilpotent} for related results.\label{known-iii}
    \item When the sequence $(X_i, p_i)$ is non-collapsing,  N\'u\~nez-Zimbr\'on, Santos-Rodr\'iguez, and the author \cite{nsz} proved that the symmetry degree is upper semi-continuous. That is, if $ \vol (B_1(p_i)) \geq v$ for some $v >0$, then 
    \[       \dim (G) \geq \limsup_{i\to \infty} \dim (G_i)    .  \]  
    \label{known-iv} 
\end{enumerate}

The main result of this paper unifies the nilpotency of  \eqref{known-iii} and the dimension upper semi-continuity of \eqref{known-iv} into a single result.

\begin{mainthm}\label{thm:main}
Let  $(X_i,p_i)$  be a sequence of pointed complete $n$-dimensional Riemannian manifolds with $\ric (X_i) \geq - (n-1)$,  $G_i \leq \iso (X_i)$ a sequence of closed groups of isometries, and assume
\begin{equation}\label{eq:egh}
(X_i,G_i, p_i) \xrightarrow{eGH} (X,G,p) .      
\end{equation}
Then $\mathfrak{g}$, the Lie algebra of $G$, admits an ideal $\mathfrak{h} \trianglelefteq \mathfrak{g}$ with $\mathfrak{g}/\mathfrak{h}$ nilpotent and such that
\begin{equation}\label{eq:h-ineq}
    \dim ( \mathfrak{h}) \leq \limsup _{i \to \infty } \dim (G_i)  .   
\end{equation}
   Moreover, if $\vol (B_1(p_i))\geq v$ for some $v > 0 $, then $\mathfrak{h}$ can be taken so that equality in \eqref{eq:h-ineq} holds.
    \end{mainthm}

\begin{Rem}
When the groups $G_i$ are discrete, \eqref{eq:h-ineq} implies that $\mathfrak{h}$ is trivial, meaning that $\mathfrak{g} = \mathfrak{g} / \mathfrak{h}$ is nilpotent, and so is the identity component of $G$.  When $\vol (B_1(p_i))\geq v$ for some $v > 0 $, equality in \eqref{eq:h-ineq} implies  
\[ \dim (G) \geq \dim (\mathfrak{h}) =  \limsup_{i \to \infty} \dim (G_i), \] 
recovering \eqref{known-iv}. On one hand,  this means that Theorem \ref{thm:main} generalizes both the nilpotency of  \eqref{known-iii} and the dimension upper semi-continuity of \eqref{known-iv}. On the other hand, we note that the proof of Theorem \ref{thm:main} uses both \eqref{known-iii} and \eqref{known-iv}. 
\end{Rem}

\begin{Rem}
In the particular case when the maps $\phi_i : G_i \to G$ demonstrating the convergence \eqref{eq:egh} can be taken to be continuous group homomorphisms (for instance, in the setting studied in \cite{grove-karcher-ruh, harvey, alattar}), Theorem \ref{thm:main-good} follows easily from  the arguments in  \cite{breuillard-gelander}. 

However, this is rarely the case when the limit group is not compact. For example, for any Lie group $G$ and any Riemannian metric $g_0 $ on $G$, if $X_i : = (G, i g_0)$ for each $i \in \mathbb{N}$, then $(X_i , G , e) \xrightarrow{eGH} (\mathbb{R}^n , \mathbb{R}^n , 0)$ with $n = \dim (G)$, but many Lie groups $G$ do not admit non-trivial group homomorphisms $G \to \mathbb{R}^n$.

\end{Rem}

\subsection{Good approximations} The main tool in the proof of Theorem \ref{thm:main} is good approximations,  introduced in \cite{nsz} and inspired by \cite{turing, hrushovski}. They allow us to study the maps $\phi _ i : G_i \to G$ demonstrating the equivariant Gromov--Hausdorff convergence without making explicit reference to the spaces the groups act on. 

\begin{Not}
For a subset $A$ of a  group  $G$, we denote by $A^n$ the set of elements $a_1 \cdots a_n \in G$ with $a_j \in A$ for each $j$. We say that $A$ is \emph{symmetric} if it contains the identity and is closed under inverses.   
\end{Not}

\begin{Def}[Good approximations]\label{def:ga-old}
    Let $G_i$ be a sequence of locally compact Hausdorff groups and $G$ a locally compact Hausdorff group. We say a sequence of functions $\phi _ i : G_i \to G$ consists of \emph{good approximations} if there are symmetric, open, pre-compact sets $A_i \subset G_i$, $A \subset G$ such that:
    \begin{enumerate}[label=\Roman*]
        \item (Almost surjectivity) For all $U \subset A $ open nonempty, one has $\phi_i (A_i) \cap U \neq \emptyset$ for $i$ large enough.\label{item:ga-1}
        \item (No expansion) For all $V \subset G$ open with $\overline{A} \subset V$, one has $\phi_i (A_i) \subset V$ for $i$ large enough.\label{item:ga-2}
        \item (Almost homomorphism) For each $n \in \mathbb{N}$ and each identity neighborhood $U \subset G$, there is $i_0 \in \mathbb{N}$ such that if $i \geq i_0 $ and   $g, h \in A_i ^n $, then 
        \[ [ \phi_i (gh)^{-1} \phi _i (g) \phi_i (h) ]  \in U    . \] \label{item:ga-3}
        \item (No compression) For each $n \in \mathbb{N}$ and each compact $K\subset A$, one has 
        \[    \phi_i ^{-1} (K) \cap A_i^n \subset A_i        \] 
        for $i$ large enough. \label{item:ga-4} 
        \item (Almost continuity) For each identity neighborhood $U \subset G$, there is a sequence of identity neighborhoods $U_i \subset G_i$ with $\phi_i (U_i) \subset U $ for $i$ large enough.\label{item:ga-5}
    \end{enumerate}
    If properties (I--V) hold, we call the sets $A_i$ and $ A$ \emph{regular neighborhoods} with respect to the approximations $\phi_i$. 
\end{Def}

In the setting of good approximations, there is a notion of small subgroups, which by \cite[Theorem C]{nsz} coincides with the geometric notion of small subgroups defined in  \cite{pan-rong}. Both of these notions are inspired by the term \emph{small subgroups} from the theory of topological groups \cite{gleason}.

\begin{Def}[Small subgroups]\label{def:ss2}
    Let $\phi_i : G_i \to G$ be a sequence of good approximations with regular neighborhoods $A_i \subset G_i$. We say a sequence of subgroups $H_i \leq G_i$ is \emph{small} or \emph{consists of small subgroups} with respect to the pairs $(\phi_i, A_i)$  if $H_i \subset A_i$ for $i$ large enough, and $ \phi_i (h_i) \to e $ for any sequence $h_i \in H_i . $

    We say the sequence $G_i$  has the \emph{no small subgroup} (NSS) property or \emph{doesn't admit small subgroups} with respect to the pairs $(\phi_i, A_i)$ if any sequence of small subgroups $H_i \leq G_i$ with respect to the pairs $(\phi_i, A_i)$ is eventually trivial. 
\end{Def}

\begin{Rem}
When no confusion can arise, we omit the dependence on the pairs $(\phi_i, A_i)$ when talking about sequences of small subgroups.    
\end{Rem}

The following result confirms that maps demonstrating  equivariant Gromov--Hausdorff convergence are good approximations. 

\begin{Thm}\label{thm:egh-to-ga} \cite{nsz, pan-rong}.   Let $(X_i, p_i)$, $(X,p)$ be pointed proper metric spaces and $G_i \leq \iso (X_i)$, $G \leq \iso (X)$ closed groups of isometries such that 
    \[     (X_i, G_i, p_i) \xrightarrow{eGH} (X,G,p) .  \]
    Then the maps $\phi_i : G_i \to G$ demonstrating the  convergence are good approximations.  Moreover, if $(X_i, p_i)$ are $n$-dimensional Riemannian manifolds with $\ric (X_i) \geq - (n-1)$ and $\vol (B_1(p_i)) \geq v$ for some $v>0$, then the regular neighborhoods can be taken so that the  sequence $G_i$ has the NSS property. 
\end{Thm}

Using Theorem \ref{thm:egh-to-ga}, Theorem \ref{thm:main} becomes an easy consequence of the following result about good approximations.

\begin{Thm}\label{thm:main-good}  Let $G_i$ be a sequence of Lie groups and  $\phi _i : G_i \to G$ a sequence of good approximations. If $G$ is a Lie group, then its Lie algebra $\mathfrak{g}$ admits an ideal $\mathfrak{h} \trianglelefteq \mathfrak{g}$ with $\mathfrak{g}/ \mathfrak{h}$ nilpotent and such that 
    \begin{equation}\label{eq:h-ineq-2}
         \text{dim} (\mathfrak{h}) \leq \limsup_{i \to \infty} \text{dim} (G_i) .  
    \end{equation}
    Moreover, if the sequence $G_i$ has the NSS property, then $\mathfrak{h}$ can be taken so that equality in \eqref{eq:h-ineq-2} holds.
\end{Thm}

\subsection{Perfectness and continuity of dimension}
As an application of our main theorems, we obtain restrictions on the types of symmetries that can converge to a given group action. Recall that a Lie algebra $\mathfrak{g}$ is called \emph{perfect} if $\mathfrak{g} = [\mathfrak{g}, \mathfrak{g}]$ and a Lie group $G$ is called \emph{topologically perfect} if $G = \overline{[G,G]}$.

It is well known that if a sequence of abelian (resp. solvable or nilpotent of fixed step) groups converge to a limit group, then the limit group is also abelian (resp. solvable or nilpotent). The following results  mirror this behavior for perfect groups.

\begin{Cor}\label{cor:no-nilpotent}
Let  $(X_i,p_i)$  be a sequence of pointed complete $n$-dimensional Riemannian manifolds with $\ric (X_i) \geq - (n-1)$ and $\vol (B_1(p_i)) \geq v$ for some $v>0$. Let  $G_i \leq \iso (X_i)$ be a sequence of closed groups of isometries such that 
\[  (X_i,G_i, p_i) \xrightarrow{eGH} (X,G,p) .   \]
If the Lie algebra of $G$ is perfect, then for $i$ large enough the identity component of $G_i$ is topologically perfect and  $ \dim (G_i ) =  \dim (G) $.
\end{Cor}

\begin{Cor}\label{cor:no-nilpotent-2}
 Let $(X_i, p_i)$, $(X,p)$ be pointed proper metric spaces and $G_i \leq \iso (X_i)$, $G \leq \iso (X)$ closed groups of isometries such that 
    \[     (X_i, G_i, p_i) \xrightarrow{eGH} (X,G,p) .  \]
    Assume  $G$ is a Lie group whose Lie algebra is perfect.   Then there is a sequence of open subgroups $O_i \leq G_i$ and a sequence of small compact normal subgroups $H_i \trianglelefteq O_i $ such that for $i$ large enough, the identity component of $O_i / H_i$ is a topologically perfect Lie group of dimension $  \dim (G)$. 
\end{Cor}

\begin{Rem}\label{rem:small}
    In Corollary \ref{cor:no-nilpotent-2}, the groups $H_i$ being small means (as defined in \cite{pan-rong}) that the orbits $H_i \cdot  p_i $ are uniformly bounded and $(X_i, H_i, p_i ) \xrightarrow{eGH}(X,\{ e\} , p)$.  
\end{Rem}

\subsection{Strategy}\label{sec:outline}

Assuming nothing goes wrong, the proof of Theorem \ref{thm:main-good} is as follows:

\textbf{Naive step 1:} By the results from \cite{nsz},  we can assume that the sequence $G_i$ has the NSS property. After passing to a subsequence, we can further assume that $\dim (G_i)$ does not depend on $i$, and $\dim (G_i) \leq \dim (G)$. 

\textbf{Naive step 2:} Let $(G_i)_0 \trianglelefteq  G_i$ denote the identity  component.  Then after passing to a subsequence, the sets $\phi_i ((G_i )_0) \subset G $ converge to a closed normal subgroup $H \trianglelefteq G$ of dimension  $\dim (G_i)$ with Lie algebra $\mathfrak{h} \trianglelefteq \mathfrak{g}$. 

\textbf{Naive step 3:} The good approximations $\phi _i : G_i \to G$ descend to good approximations $\psi _i : G_i / (G_i)_0 \to G/ H$.   Since $G_i / (G_i)_0$ is discrete, by \cite{breuillard-green-tao} we conclude that the identity component of  $G/H$ (and consequently $\mathfrak{g} / \mathfrak{h}$) is nilpotent. 

While in spirit the proof of Theorem \ref{thm:main-good} follows the above outline, \textbf{Naive step 2} can fail due to $(G_i)_0$ being ``too large'', as in the following example.

\begin{Exa}\label{ex:bad-approximation}
    Let $G_i : = \mathbb{R}$ for each $i$, $G : = \mathbb{R}^2 $, and $\phi_i : G_i \to G$ be given by
    \[    \phi_i (x) : = \left( \, x - \left\lfloor   x / i ^2  +  1 / 2       \right\rfloor i^2  \,  ,   \left\lfloor   x / i ^2 + 1/ 2 \right\rfloor / \, i \,                        \right) .   \]
    Also let 
    \begin{align*}
       A_i &  : =  \bigcup_{j = - i } ^i  \,( \, i^2 j - 1  \, , \,  i^2 j + 1  \, )    , \\
       A & : = ( -1, 1 ) \times  ( -1,1 )    .    \end{align*}
    It is then not hard to check that the maps  $\phi_i$ are good approximations with regular neighborhoods $A_i \subset G_i$ and $A \subset G$.   Note that $(G_i)_0 = G_i$ and the sequence $\phi_i (G_i)$ saturates $G$ (or more precisely, the sequence $\phi_i (G_i)$ converges to $G$ in the pointed Hausdorff/Attouch--Wets topology), so step 2 of the above strategy fails.

     In order to avoid that, we need to restrict ourselves to identity neighborhoods of appropriate size. In other words, we  consider
    \[     H_i : =  ( -1 , 1 )  \subset G_i     .       \]    
    Then the sets $\phi _i (H_i) \subset G$ converge to the set 
    \[     H : =  ( -1, 1 ) \times \{ 0 \}  \subset G   .        \]
    Note that the quotients $G_i / \langle H_i \rangle $ are trivial, so we cannot obtain good approximations  $\psi _i : G_i / \langle  H_i \rangle \to G / \langle H \rangle \cong \mathbb{R} $. Therefore, we need to consider ``local quotients'' using the theory of local groups as in \cite{goldbring, breuillard-green-tao}.   Namely, if
    \begin{align*}
       W_i &  : =  \bigcup_{j = - \lfloor \frac{i}{10}\rfloor } ^{\lfloor \frac{i}{10} \rfloor }  \,\left( \, i^2 j - \frac{1}{10}  \, , \,  i^2 j + \frac{1}{10}   \, \right)    , \\
       W & : = \left( - \frac{1}{10} , \frac{1}{10} \right)  \times  \left( - \frac{1}{10} , \frac{1}{10} \right)    ,
       \end{align*}
    then one can form the quotients $W _i / H_i$ and $W / H$ in a meaningful way and construct the desired ``local'' good approximations $\psi _i : W_i / H_i \to W / H$. 
\end{Exa}

\begin{Rem}
    While the above example is purely algebraic, this phenomenon naturally occurs in the setting of equivariant Gromov--Hausdorff convergence.  For instance, consider a sequence of flat cylinders converging to $\mathbb{R}^2$ by making the circle factor larger and larger, equipped with a sequence of $\mathbb{R}$-actions that converge to the standard $\mathbb{R}^2$-action on itself (compare \cite[Example 1]{kapovitch-wilking}). 
\end{Rem}

Using the above example as a guide, the proof of Theorem \ref{thm:main-good} is as follows: 

\textbf{Correct step 1: }Same as in previous strategy.

\textbf{Correct step 2: }Find suitable identity neighborhoods $H_i \subset G_i$ that converge to a subset $H \subset G$ with a well defined Lie algebra $\mathfrak{h} \leq \mathfrak{g}$  of dimension $\dim (G_i)$.  

\textbf{Correct step 3: }Find small identity neighborhoods $W_i \subset G_i$ and $W \subset G$ in such a way that the quotients $W_i / H_i$ and $W / H$ are well defined and the good approximations $\phi _i$ descend to  local good approximations $\psi _i : W_i / H_i \to W / H$.  Since $W_i / H_i$ is discrete, by  \cite{breuillard-green-tao} we can conclude that $W / H$ (and consequently $\mathfrak{g} / \mathfrak{h}$) is nilpotent.

\subsection{Outline}  This paper is organized as follows. In Sections \ref{sec:prelims} and \ref{sec:lie}, we discuss the background material. In Section \ref{sec:local-good}, we prove that under suitable conditions,  good approximations descend to quotients, enabling \textbf{Correct step 3} above. In Section \ref{sec:to-nss}, we perform \textbf{Correct step 1}, reducing Theorem \ref{thm:main-good} to the case when the groups in the sequence have the NSS property. 

In Section \ref{sec:ops-convergence}, we study how, under the NSS condition, one-parameter subgroups of the groups $G_i$ converge to one-parameter subgroups of $G$. In Section \ref{sec:h-construction}, we construct the sets $H_i \subset G_i$ and $H \subset G$ mentioned in \textbf{Correct step 2}, and in Section \ref{sec:properties-of-h} we show that $H$ satisfies the desired properties, finishing the proof of Theorems \ref{thm:main-good} and \ref{thm:main}.

Lastly, in Section \ref{sec:proofs-of-corollaries} we prove Corollaries \ref{cor:no-nilpotent} and \ref{cor:no-nilpotent-2}.

\section{Preliminaries}\label{sec:prelims}

\subsection{Notation}

In a topological space $X$, if $A \subset X$, we denote by $\inte (A)$ the interior of $A$, and by $\overline{A}$ the closure of $A$. Throughout this paper, $\dim (\cdot )$ denotes topological dimension. 

 In a metric space $X$, we denote by $B_r^X(x)$ the open ball with center $x$ and radius $r$.  If the metric space is clear from the context, we write $B_r(x)$  instead.

For a set $A$ and $n \in \mathbb{N}$, we denote bt $A^{\times n} $ the $n$-fold cartesian product $A \times \cdots \times A$. If $A$ is a subset of a group $G$, we denote by $A^n$ the product set
\[      A^n  : = \{  a _1 \cdots a_n \in G \, \vert \, a_j \in A   \text{ for all }j \}.            \]

If $f: M \to N$ is a smooth map between smooth manifolds, its differential is denoted by $f_{\ast} : TM \to TN$ or $d_p f : T_pM \to T_{f(p)}N$ depending on whether we want to emphasize the basepoint. 

If $G$ is a Lie group with Lie algebra $\mathfrak{g}$, we denote by $\exp $ and $\log$ the algebraic exponential and logarithm maps between (subsets of) them, even if $G$ is equipped with a Riemannian metric. For $g \in G$, we denote by $L_g : G \to G$ and $R_g : G \to G$ the left and right multiplication maps, respectively. The adjoint representations are denoted by $\Ad : G \to \operatorname{GL} (\mathfrak{g})$ and $\ad : \mathfrak{g} \to \operatorname{End} (\mathfrak{g}) $.

\subsection{Linear algebra} Here we review the results from linear algebra that we will use. For a Banach space $V$ and a bounded operator  $ T: V \to V$,  we denote by $\Vert T \Vert $ the operator norm of $T$ and by $e^T : V \to V$ the map given by
\[      e^T := \sum_{k = 0}^{\infty} \frac{T^k}{k!}.                 \]

\begin{Pro}  \label{pro:rudin} 
Let $V $ be a Banach space and  $T : V \to V$ a bounded operator. If $ \Vert T - \Id _V \Vert  \leq \frac{1}{2} $, then $T$ is invertible and 
\[    \Vert T ^{-1}  - \Id_V  \Vert  \leq 2 \,   \Vert T - \Id_V \Vert .           \]    
\end{Pro}

\begin{proof}
    Use the expansion $T^{-1} -  \Id_V = -  (   T  -\Id_V ) + (  T -\Id_V )^2 - (T-\Id_V)^3 + \ldots $.
\end{proof}

\begin{Lem}  \label{lem:exponentials} 
Let $V $ be a Banach space and  $ T _1 , \ldots , T_k  : V \to V$  bounded operators.  Assume
\[ S : =  \sum_{i = 1 }^k \Vert T_i \Vert < 1 .      \] 
Then 
\begin{equation}\label{eq:exponential-bound}
  \Vert e^{T_1} e^{T_2} \cdots e^{T_k} - \Id _V \Vert \leq 2 S e ^S  .                      
\end{equation}
\end{Lem}

\begin{proof}
This can be easily obtained from the telescopic expansion
\[  e^{T_1} e^{T_2} \cdots e^{T_k} - \Id _V = \sum_{i = 1 } ^k  [  e^{T_1} e^{T_2} \cdots e^{T_i} - e^{T_1} e^{T_2} \cdots e^{T_{i-1}} ].             \]
\end{proof}

The following result, known as John's Ellipsoid Theorem,  states that any norm can be estimated by one induced by an inner product. A proof can be found for example in \cite{milman-schechtman}. 

\begin{Thm}\label{thm:john}
   Let $V$ be an $n$-dimensional real vector space and $\Vert \cdot \Vert : V \to \mathbb{R}$ a norm. Then there is a norm $\Vert \cdot \Vert ' : V \to \mathbb{R}$ induced by an inner product and satisfying 
    \[      \Vert \vv \Vert  \leq \Vert \vv \Vert ' \leq C(n) \Vert \vv \Vert          \]
    for all $\vv \in V$. 
\end{Thm}

\subsection{Hausdorff convergence} In this section, we review the Hausdorff convergence in the space of subsets of a metrizable topological space.

\begin{Def}
Given a metric  space $(X,d)$ and two subsets $A,B \subset X$, the \emph{Hausdorff distance} between them (denoted $d_H(A,B)$) is defined to be
\[    \max\{  \sup \{  \inf \{ d(a,b) \, \vert \, b \in B \} \, \vert \, a \in A \} ,    \sup \{  \inf \{ d(a,b) \, \vert \, a \in A \} \, \vert \, b \in B \}      \} .         \]
We say that a sequence of subsets $A_i \subset X$ \emph{converges in the Hausdorff topology} to a subset $A \subset X$ if $d_H(A_i , A) \to 0$ as $i \to \infty$. 
\end{Def}

If $X$ is  a metrizable topological space, we say that a sequence $A_i \subset X$ \emph{converges in the Hausdorff topology} to $A \subset X$ if after equipping $X$ with a compatible metric $d : X \times X \to \mathbb{R}$, the sequence $A_i$ converges in the Hausdorff topology to $A$ in the above sense. In general, this definition depends on the metric $d$, but if there is a compact set $K \subset X$ with 
\begin{equation}\label{eq:compact-hausdorff-sequence}
\bigcup _{i \in \mathbb{N}} A_i \subset K ,    
\end{equation}
then this definition does not depend on $d$. Moreover, assuming \eqref{eq:compact-hausdorff-sequence}, the sequence $A_i $ converges to $A $ if and only if the two following conditions hold:
\begin{itemize}
    \item For any $a \in A$, there is a sequence $a_i \in A_i$ with $a_i \to a$.
    \item Given a sequence $a_i \in A_i$, any accumulation point belongs to $\overline{A}$. 
\end{itemize}

 \begin{Rem}
    Often in the literature, the  Hausdorff topology is restricted to the set of closed subsets of a given ambient space $X$. However, it will be convenient for us to consider non-closed subsets as well. As a consequence of this, Hausdorff limits of sequences may not be unique. Nevertheless, two distinct closed sets cannot be Hausdorff limits of the same sequence. 
\end{Rem}

The Hausdorff topology also enjoys a natural compactness property: if a sequence of non-empty subsets $A_i \subset X$ satisfies \eqref{eq:compact-hausdorff-sequence} for some compact subset $K \subset X$, then after passing to a subsequence, the sequence $A_i$ converges in the Hausdorff topology to a closed subset $A \subset K$.

\subsection{Properties of good approximations}

In this section we review the basic properties of good approximations.  We begin with a compactness property.

\begin{Lem}\label{lem:convergent-sequences}
    Let $\phi_i: G_i \to G$ be good approximations with $A_i \subset G_i$, $A\subset G$ as regular neighborhoods, and assume $G$ is metrizable. Then for any $n \in \mathbb{N}$ and any sequence $g_i \in A_i ^n$, after passing to a subsequence, there is $g \in \overline{A^n}$ with $\phi_i (g_i) \to g$. Moreover, if $g_i \in \partial A_i$ for all $i$, then $g \in \partial A$. 
\end{Lem}

\begin{proof}
    The first part is proven in \cite[Lemma 4.12]{nsz}. To verify the last claim, assume by contradiction that $g_i \in \partial A_i$ for all $i$, but $g \notin \partial A$. 
\begin{center}
    \textbf{Case I: }$g \in A$. 
\end{center}
    Pick a compact neighborhood $K \subset A$ of $g$. Then $\phi_i (g_i) \in K$ for $i$ large enough so by \eqref{item:ga-4} we get $g_i \in A_i$; a contradiction.
\begin{center}
    \textbf{Case II: }$g \notin \overline{A}$. 
\end{center}
    Pick an identity neighborhood $U \subset G$ with $g \overline{U} \cap \overline{A} = \emptyset$.   By  \eqref{item:ga-5}, there is a sequence of identity neighborhoods $U_i \subset A_i$ with $\phi_i (U_i) \subset U$ for all $i$ large enough.  Then choose $h_i \in A_i$ with $g_i ^{-1 } h_i   \in U_i$ for all $i$. After passing to a subsequence we can assume $\phi_i (g_i ^{-1 } h_i  ) \to u $ for some $u  \in \overline{U}$. Using \eqref{item:ga-3}, we get
    \[      \phi _i (h_i ) = \phi_i ( (g_i) (g_i ^{-1} h_i ) ) \to   g u .          \]
    Since $gu \notin \overline{A}$, this contradicts \eqref{item:ga-2}.
\end{proof}

The following lemma allows one to approximate sets in the limit group by sets in the sequence.

\begin{Lem}\cite{nsz}\label{lem:approximation-by-internal-sets}
    Let $\phi_i : G_i \to G$ be a sequence of good approximations with regular neighborhoods $A_i \subset G_i $ and $A \subset G$.  For all $K \subset U \subset A$ with $K$ compact and $U$ open, there are open sets $A_i' \subset A_i  $ with 
        \[ \phi_i^{-1}(K) \cap A_i \subset A_i' \subset \phi_i^{-1}(U)                 \]
        for $i$ large enough.  Moreover, if $U$ is symmetric, the sets $A_i'$ can be taken to be symmetric.
\end{Lem}

The following result allows one to ``localize'' a good approximation.

\begin{Pro}\label{pro:zoom-i}
 \cite{nsz}.  Let $\phi_i: G_i \to G$ be good approximations with $A_i \subset G_i$, $A\subset G$ as regular neighborhoods, and assume $G$ is metrizable. For any open symmetric set $B \subset A$, there are open symmetric sets $B_i \subset A_i$ such that the maps $\phi_i $ are good approximations with regular neighborhoods $B_i$ and $B $. Moreover, a sequence of groups $H_i \leq G_i$ is small with respect to the pairs $(\phi _i ,  A_i)$  if and only if it is small with respect to the pairs $(\phi_i, B_i)$. 
\end{Pro}

The following theorem establishes that when the limit group $G$ is Lie, the sequence $G_i$ admits a sequence $H_i \leq G_i$ of maximal small subgroups.

\begin{Thm}\label{thm:largest-small} \cite{nsz}.
    Let  $\phi _i : G_i \to G$ be a sequence of good approximations with regular neighborhoods $A_i \subset G_i$ and $A \subset G$. If $G$ is a Lie group, then there is a sequence of subgroups $H_i  \leq G_i $ such that:
    \begin{itemize}
        \item The sequence $H_i$ is small.
        \item $H_i \trianglelefteq \langle A_i \rangle $ for $i$ large enough.
        \item For any sequence $H_i ' \leq G_i$ of small subgroups, one has $H_i ' \leq H_i$ for $i$ large enough. 
        \item $G_i ' : = \langle A_i \rangle / H_i$ is a Lie group for $i$ large enough.
    \end{itemize}
\end{Thm}

The following result establishes that in the NSS setting, the dimension of the groups is upper semi-continuous.

\begin{Thm}\label{thm:nsz} \cite{nsz}.
    Let $G_i$ be a sequence of Lie groups and  $\phi _i : G_i \to G$ a sequence of good approximations. If $G$ is a Lie group and  the sequence $G_i$ has the NSS property, then
    \[    \text{dim} (G) \geq \limsup_{i \to \infty} \text{dim} (G_i) .    \] 
\end{Thm}

\subsection{Escape norm}
Inspired by the work of Gleason  \cite{gleason}, 
the escape norm was introduced in  \cite{breuillard-green-tao, carolino}.
\begin{Def}[Escape norm]
Let $G $ be a group and $B \subset G$ an open symmetric set. The \emph{escape norm} $\Vert \cdot \Vert _{B} : G \to \mathbb{R}$  is defined as
\[    \Vert g \Vert _B : =  \inf \left\{ \frac{1}{m+1} \,\, \Big| \,\, g^j \in B  \text{ for all } j \in \{ 0, 1, \ldots , m \} \right\}  .                \]
If $G$ is a Lie group with Lie algebra $\mathfrak{g}$, we can also define $\vert \cdot \vert _B  : \mathfrak{g} \to \mathbb{R}$ as:
\[  \vert \vv \vert _B :   = \frac{1}{\tau _B (\vv)}  ,   \]
where $\tau _B (\vv) : = \inf \{ t > 0  \vert \exp (t \vv) \notin B \} \in (0, \infty ]$.  
\end{Def}

\begin{Rem}
In general, $\Vert \cdot \Vert _B : G \to \mathbb{R}$ is not continuous and  $\vert \cdot \vert _B : \mathfrak{g} \to \mathbb{R}$ is not necessarily a norm.  However,  $\vert \cdot \vert _B : \mathfrak{g} \to \mathbb{R}$ is positively homogeneous and in certain instances it can be approximated by an actual norm (see Lemma \ref{lem:norm-approximates} and \cite[Chapter 8]{carolino}).
\end{Rem}

In general, the Lie algebra escape norm bounds the group escape norm.

\begin{Pro}\cite{nsz}.\label{pro:norm-relation}
    For all $\vv \in \mathfrak{g}$ one has 
    \begin{equation}\label{eq:escape-less-than-norm}
         \Vert \exp (\vv) \Vert _{B}     \, \leq  \, \vert \vv \vert _{B}  . 
    \end{equation}
\end{Pro}

The following inequalities, due to Breuillard, Green, and Tao,  are known as the Gleason Lemmas. They follow from \cite[Theorem 8.1]{breuillard-green-tao} (see also \cite[Theorem 7.7]{carolino} and \cite[Theorem 7.1]{nsz}).

\begin{Thm}\label{thm:gleason} \cite{breuillard-green-tao}.     Let $\phi_i : G_i \to G$ be a sequence of good approximations. Assume $G$ is a Lie group and equip its Lie algebra $\mathfrak{g}$ with an inner product.  Set  $B   : = \exp (B_r^{\mathfrak{g}}(0)) \subset G$ and let $B_i \subset G_i$ be the open symmetric sets given by Proposition  \ref{pro:zoom-i}. Then there is $C_0 \geq 1 $ such that if $r > 0 $ is small enough, for sufficiently large $i$ one has:
    \begin{itemize}
        \item $ \Vert g_1 \cdots g_m \Vert _{B_i} \leq C_0 \sum_{j =1}^m \Vert g_j \Vert _{B_i}   $ for all $ g_1, \ldots , g_m \in G_i.$
        \item $\Vert ghg^{-1} \Vert_{B_i} \leq 1000 \, \Vert h \Vert _{B_i} $ for all $g,h \in B_i^{10}$.
        \item $\Vert [g,h] \Vert_{B_i} \leq C_0 \, \Vert g \Vert _{B_i} \Vert h \Vert _{B_i} $ for all $g,h \in B_i^{10}$.
    \end{itemize}
\end{Thm}

The following lemma provides a situation when  the Lie algebra escape norms can  be approximated by actual norms. 

\begin{Lem}\label{lem:norm-approximates}
    \cite{nsz, carolino}. Let $G_i,$ $G$, $\phi_i$, $B$, $B_i$, $C_0$ be as in Theorem \ref{thm:gleason}. Assume the sequence $G_i$ has the NSS property and  consists of Lie groups, and let $\mathfrak{g}_i$ be the Lie algebra of $G_i$. Then for $i$ large enough there is a genuine norm $\Vert \cdot \Vert _{i} : \mathfrak{g}_i \to \mathbb{R}$ with
    \begin{equation}\label{eq:norm-approximates}
     \vert \vv \vert _{B_i} \leq \Vert \vv \Vert _i \leq 2C_0 \vert \vv \vert _{B_i}       
    \end{equation}
    for all $\vv \in \mathfrak{g}_i$. 
\end{Lem}

\subsection{Local groups} In this section, we review the basics of local groups. We refer the reader to \cite{goldbring} and \cite[Appendix B]{breuillard-green-tao} for further discussion.

\begin{Def}[Local groups]
    A \emph{local group} is a pointed topological space $(G, e )$ equipped with an open set $\Omega \subset G \times G $ and continuous maps $m : \Omega \to G$ and $ ( \cdot ) ^{-1} : G \to G$ satisfying the following:
    \begin{itemize}
        \item For all $g \in G$, one has $(g,e), (e, g)\in \Omega$ and  
        \[ m( g , e)  = m(e  ,g) = g.\]
        \item For all $g \in G$, one has $(g,g^{-1}), (g^{-1}, g)\in \Omega$ and
        \[ m (g , g^{-1} )= m (g^{-1} , g ) = e .\]
        \item If $g,h,k \in G$ are such that $(g,h), (m(g , h ) , k), (h,k), (g,m (h , k)) \in \Omega$, then 
        \[    m   (m(g , h) , k) = m( g ,m (h , k)) .       \]
    \end{itemize}
    When no confusion arises, we denote the local group $(G,e , \Omega, m , ( \cdot  )^{-1} )$ simply by $G$. 
\end{Def}

\begin{Def}[Well defined products]
    For $g_1, \ldots , g_n \in G$, we say the product $g_1 \cdots g_n$ is \emph{well defined} if for each $1 \leq i \leq j \leq n $ there is an element $g_{[i,j]} \in G$ such that
    \begin{itemize}
        \item For each $1 \leq i \leq n$, one has  $g_{[i,i]} = g_i$.
        \item If $1 \leq i \leq j < k \leq n$, one has $(g_{[i,j]}, g_{[j+1, k]}) \in \Omega$ and 
        \[    g_{[i,k]} = m  (g_{[i,j]}, g_{[j+1, k]})   .   \]
    \end{itemize}
    In such a case, the element $g_{[1,n]} $ is uniquely defined and is denoted by $g_1\cdots g_n$. If in addition one has  $g_1 = \cdots = g_n = g$, the product $g_1 \cdots g_n $ is denoted by $g^n$. 
    
    For a subset $U \subset G$, by an abuse of notation we write $g_1 \cdots g_n \in U$ to denote the assertion that the product $g_1 \cdots g_n$ is well defined and it is in $U$.

    For subsets $W , $ $ U \subset G$ and $n \in \mathbb{N}$, we also  write $W ^n \subset U$ to denote the assertion that for any $g_1, \ldots , g_n \in W$ we have $g_1 \cdots g_n \in  U$. 
\end{Def}

\begin{Def}[Morphism]
    Let $G$ and $H$ be local groups. A \emph{morphism} between $G$ and $H$ is a function $\pi : G \to H$ with the property that if  $g_1g_2$ is well defined for some $g_1, g_2 \in G$, then $\pi (g_1)\pi(g_2)$ is well defined and equals $\pi (g_1g_2)$.  
\end{Def}

\begin{Def}[Multiplicative set]
    Let $G$ be a local group and $A \subset G$. We say $A$ is a \emph{multiplicative set} if $(A \cup A^{-1} )^{200} \subset G$. 
\end{Def}

\begin{Def}[Sub-local group]
    Given a local group $G$, a symmetric subset $H \subset G$ is called a \emph{sub-local group} if there is an open symmetric set $V \subset G$  such that 
\begin{itemize}
    \item $H \subset V$ and $H$ is closed in $V$.
    \item Whenever $h_1, h_2 \in H$ and $h_1  h_2 \in V$, one has $h_1 h_2 \in H$.
\end{itemize}
In such a case, $V$ is called an \emph{associated neighborhood} for $H$. 
\end{Def}
\begin{Def}[Normal sub-local group]
We say that a sub-local group $H$ is  \emph{normal} if it admits an associated neighborhood $V$ such that  
    \begin{itemize}
        \item Whenever $h \in H$, $a \in V $, and $a h a^{-1} \in V$, one has $a h a^{-1} \in H$. 
    \end{itemize}
    In such a case, $V$ is called a \emph{normalizing neighborhood} for $H$. 
\end{Def}

\begin{Pro}[Local group quotient] \label{pro:local-quotient}
 Let $G$ be a local group, $H \subset G $ a normal sub-local group with normalizing neighborhood $V$, and $W \subset G$ an open symmetric set. Assume $W^{6} \subset V$ and $V^6 \subset G$. Then
 \begin{itemize}
     \item The binary relation defined by 
     \[  x \sim y \text{ if and only if  } x^{-1} y \in H          \]
     is an equivalence relation on $W$.  
     \item Setting  
     \[     \Omega _{W, H } : = \{ \, ( a , b ) \in (W/\sim) ^{\times 2} \, \vert \,  xy \in  W \text{ for some }x \in a, \, y\in b \, \}   ,      \]
    the local group structure of $G$ induces a well defined inverse map $W / \sim  \to W / \sim $ and a well defined product $\Omega _{W,H} \to W/ \sim$, turning $W/ \sim$, equipped with the quotient topology, into a local group. 
 \end{itemize}
The local group obtained in Proposition \ref{pro:local-quotient} is denoted by   $(G/ H)_W$. 
\end{Pro}

\begin{Def}[Local Lie group]
    Let $G$ be a local group and a smooth manifold.  We say $G$ is a \emph{local Lie group} if the multiplication and inverse maps are smooth. 
\end{Def}

 For several structural results in Lie theory, their proofs are purely local, so they carry over to local Lie groups (see for example \cite{hall, knapp}).

\begin{Pro}\label{pro:lie-local}          
    Let $G$ be a local Lie group. Then 
    \begin{itemize}
        \item  The tangent space $\mathfrak{g} : = T_eG$ carries a natural Lie algebra structure.
        \item  There is an open set $S \subset \mathfrak{g}$, star-shaped around $0$, such that the exponential map $\exp : S \to G $ is well defined in the usual way. 
        \item  Any sub-local group $H \subset G$ is a submanifold with 
        \[  \mathfrak{h} : = T_e H  = \bigcup_{\varepsilon > 0 } \{ \vv \in \mathfrak{g} \, \vert \,  \exp (t \vv ) \in H  \text{ for all } t \in (- \varepsilon , \varepsilon ) \}   . \]
        Moreover, $\mathfrak{h}$ is a Lie subalgebra of $\mathfrak{g}$.
        \item  If in addition, $H$ is  normal, then $\mathfrak{h}$ is an ideal of $\mathfrak{g}$.  Moreover, $\mathfrak{g} / \mathfrak{h}$ is naturally isomorphic to the Lie algebra of the quotient local group $(G/H)_W$, where $W$ is as in Proposition \ref{pro:local-quotient}. 
    \end{itemize}
\end{Pro}

\subsection{Ultrafilters, ultraproducts, and ultralimits}
In this section, we review the basics of ultrafilters and ultralimits.  We refer the reader to \cite[Section 2]{gromov-asymptotic} and \cite[Appendix A]{breuillard-green-tao} for further discussion.  We denote by $\mathcal{P}(\mathbb{N}) $ the power set of the natural numbers.

\begin{Def}[Ultrafilter]
    We say $\alpha : \mathcal{P}(\mathbb{N}) \to \{ 0, 1 \}$ is an \emph{ultrafilter} if $\alpha (\mathbb{N}) = 1$ and  $  \alpha (A \cup B) = \alpha (A) + \alpha (B)  $ whenever $A \cap B = \emptyset$. We say that an ultrafilter is \emph{non-principal} if $\alpha (F) = 0$ whenever $F $ is finite. 
\end{Def}

Using the axiom of choice, one can show that non-principal ultrafilters exist. For the remainder of this paper, we fix a non-principal ultrafilter and call it $\alpha$. 

\begin{Def}[Ultraproduct]
    Let $A_i$ be a sequence of sets. The \emph{ultraproduct} of the sequence $A_i$ is defined as the quotient
    \[   \prod_{i \to \alpha} A_i : =    \left( \prod _{i \in \mathbb{N}} A_i \right) / \sim  ,     \]
    where two sequences $(a_i)$ and $(b_i)$ in $\prod_{i\in \mathbb{N}}A_i$ are $\sim$-related if $\alpha (\{ i \in \mathbb{N} \, \vert \, a_i = b_i \}  ) = 1$.  
\end{Def}

\begin{Rem}
Given a sequence of local groups $G_i$, the ultraproduct   $ \prod _{i \to \alpha } G_i ,   $ equipped with the discrete topology, is naturally a local group. Moreover,  given multiplicative sets $A_i \subset G_i$, the ultraproduct  $\prod_{i \to \alpha }A_i $  is naturally a multiplicative set in  $\prod_{i \to \alpha }G_i $.     
\end{Rem}

\begin{Def}[Ultraconvergence]
    Let $X$ be a topological space and $x_i $ a sequence of elements of $X$. We say the sequence $x_i$ \emph{ultraconverges} to a point $p \in X$ if for any neighborhood $U$ of $p$, one has $\alpha ( \{ i \in \mathbb{N} \, \vert \, x_i \in U  \} ) = 1$.  In such a case, we say that $p$ is the \emph{ultralimit} of the sequence $x_i$, and we write $ p = \lim_{i \to \alpha} x_i    . $
\end{Def}

\begin{Rem}
It is easy to see that if $X$ is a compact Hausdorff topological space, then any sequence in $X$ ultraconverges to a unique ultralimit.     
\end{Rem}

\section{Quantitative Lie theory}\label{sec:lie}

In this section, we prove a couple of quantitative versions of elementary facts about Lie groups. We give complete proofs for convenience of the reader. 

The first main result of this section states that the product of two elements in the exponential image of a small ball around the origin is in the exponential image of a somewhat larger ball.  It will be used to show that the set $H$ constructed in step 2 of the proof of Theorem \ref{thm:main-good} is a sub-local group of the limit group (see Section \ref{sec:outline}).  The proof follows the classical proof of the Baker--Campbell--Hausdorff formula (see for example \cite{hall, knapp}). 
\begin{Lem}\label{lem:bch-quant}
    Let $G$ be a Lie group with Lie algebra  $\mathfrak{g}$, and $C \geq 1$. Assume $\mathfrak{g}$ is equipped with a norm $\Vert \cdot \Vert $ and  
\begin{equation}\label{eq:small-ad}
      \Vert     [\xx , \yy ]  \Vert  \leq  C \Vert \xx \Vert \Vert \yy  \Vert  
\end{equation}
for all $\xx , \yy \in \mathfrak{g}$. Then for all $\xx , \yy \in B_{1/16C}^{\mathfrak{g}}(0)$, there is $\zz \in B_{1/2C}^{\mathfrak{g}}(0)$ with
\[     \exp (\zz) = \exp (\xx) \exp (\yy) .     \]
\end{Lem}

\begin{proof}
Using the fact that $\mathfrak{g}$ is a finite-dimensional vector space, we identify $T_{\xx} \mathfrak{g}$ with $\mathfrak{g}$ for each $\xx \in \mathfrak{g}$. Doing so,  equips $T_{\xx} \mathfrak{g}$ with the norm $\Vert \cdot \Vert $.  Consider the map 
\[ \Psi _{\xx}  : \mathfrak{g} \cong T_{\xx} \mathfrak{g}  \to \mathfrak{g}   \]
defined by 
\begin{equation}\label{eq:psi-x-definition}
      \Psi _{\xx} : = ( L _ {\exp (\xx)} ) ^{-1}_{\ast} \circ [ d_{\xx} \exp ] . 
\end{equation}
We know \cite[Section 5.4]{hall} that for each $\xx  \in \mathfrak{g}$ we  have 
\[
 \Psi _{\xx}  =     \sum_{j = 0} ^{\infty} \frac{(-1)^j  \ad _{\xx} ^j }{(j+1)!}  .    
\]
The hypothesis \eqref{eq:small-ad} means $\Vert \ad_{\xx} \Vert \leq C \Vert \xx \Vert $ for all $\xx \in \mathfrak{g}$, so we have
\begin{align*}
       \Vert \Psi_{\xx} - \Id _{\mathfrak{g}} \Vert  
       & \leq   \sum_{j = 1} ^{\infty} \frac{ \Vert  \ad _{\xx} \Vert ^j}{(j+1)!} \\
        & \leq \sum_{j = 1} ^{\infty} \frac{ C ^j \Vert \xx \Vert ^j}{(j+1)!} \\
        & \leq  \tfrac{1}{2} C \Vert \xx \Vert e ^{C \Vert \xx \Vert } .
\end{align*}  
By Proposition \ref{pro:rudin}, this shows that the exponential map is a local diffeomorphism in $ B : =  B_{1/2C} ^{\mathfrak{g}} (0) $ and 
\begin{equation}\label{eq:psi-inverse-good}
     \Big \Vert  \Psi _{\xx} ^{-1}   - \Id _{\mathfrak{g}}     \Big \Vert \leq  2 C \Vert \xx \Vert    
\end{equation}
for all $\xx \in B$. 

Fix $\yy \in B_{1/16C}^{\mathfrak{g}}(0)$ and consider the left-invariant vector field $Y \in \mathfrak{X}(G)$ given by 
\[        Y (g) : = (L_g)_{\ast} (\yy).                     \]
The lift $\tilde{Y} \in \mathfrak{X}(B)$ is then given by 
\begin{equation}\label{eq:exp-related}
    \tilde{Y}  (\xx) :  = \Psi_{\xx} ^{-1} (\yy )= [ d_{\xx} \exp ] ^{-1} Y (\exp (\xx)) .
\end{equation}
From \eqref{eq:psi-inverse-good}, we know that  $\Vert \tilde{Y }(\xx) \Vert \leq \tfrac{1}{8C}$ for each $\xx \in B$.   We denote by  $ \Phi ^{\yy} : G  \times \mathbb{R}  \to G $ and $\tilde{\Phi} ^{\yy} : B_{1/4C}^{\mathfrak{g}}(0) \times (- 2, 2 ) \to B$  the flows of $Y$ and $\tilde{Y}$, respectively. Note that
\[  \Phi^{\yy}_t (g)  = g \exp (t\yy)     \]
for all $g \in G$ and $t \in \mathbb{R}$. Hence, for all $\xx \in B_{1/4C}^{\mathfrak{g}}(0)$ we have
\[     \exp ( \tilde{\Phi} _1^{\yy} (\xx)   ) = \Phi _1 ^{\yy} (\exp (\xx)) = \exp (\xx) \exp (\yy)  ,      \]
where the first equality follows from  \eqref{eq:exp-related}. Setting $\zz : = \tilde{\Phi}_1 ^{\yy}(\xx)$ finishes the proof.
\end{proof}

The second main result of this section establishes that canonical coordinates of the second kind are  locally surjective onto a ball of definite radius depending only on the absolute value of the structure constants. This lemma will be key in establishing \eqref{eq:h-ineq-2}.

\begin{Lem}\label{lem:lie-ift}
    For each $n \in \mathbb{N}$, $C\geq 1$, $\varepsilon > 0 $, there is $\delta > 0 $ such that the following holds. Let $\mathfrak{g}$ be a finite-dimensional Lie algebra equipped with an inner product, $\vv_1, \ldots , \vv_n \in \mathfrak{g}$ an orthonormal basis, and assume 
    \[
          \Vert [\xx,\yy] \Vert \leq C \Vert \xx \Vert \Vert \yy \Vert   
    \]   
    for all $\xx ,\yy \in \mathfrak{g}$. Then for each $\ww \in \mathfrak{g}$ with $\Vert \ww \Vert < \delta $, there are $t_1, \ldots , t_n \in ( - \varepsilon , \varepsilon ) $ with 
    \[   \exp (\ww )  = \exp (t_1 \vv _1) \cdots \exp (t_n \vv _n )  \in G    \]
    for any Lie group $G$ with Lie algebra $\mathfrak{g}$. In fact, one can take $\delta  : = \min \{ \tfrac{\varepsilon}{4}, \tfrac{1}{64 n^2 C } \} $. 
\end{Lem}

Near the end of the proof of Lemma \ref{lem:lie-ift}, we will need the following quantitative version of the classical Inverse Function Theorem.

\begin{Thm}[Inverse Function Theorem]\label{thm:inverse-function-theorem}
    Let $U \subset \mathbb{R}^n $ be an open set and $\Theta : U \to \mathbb{R}^n$ a differentiable function with  $\Theta ( 0 ) = 0 $ and 
    \begin{equation}\label{eq:hypothesis-good-partials}
          \Vert \tfrac{\partial }{\partial t_j} \Theta (t_1, \ldots , t_n) - \text{e} _j  \Vert \leq \tfrac{1}{2n}        
    \end{equation}
    for all $j \in \{ 1, \ldots , n \}$ and  $ (t_1, \ldots , t_n) \in  U$. Assume $r > 0 $ is such that $[-r,r ] ^{\times n} \subset U$. Then for each $x\in B_{r/4}(0)$ there are unique  $t_1, \ldots , t_n \in (-r, r)$ such that 
    \[      \Theta ( t_1, \ldots , t_n ) = x.                       \]
\end{Thm}
\begin{proof}
    Using \eqref{eq:hypothesis-good-partials}, it is not hard to show that $\Vert \Theta (z) \Vert \geq r / 2 $ for any $z \in \partial ( [-r,r] ^{\times n})$. Then for any $x \in B_{r/4}(0)$, the function $f : [-r,r]^n \to \mathbb{R}$ given by $f(y) : = \tfrac{1}{2} \Vert \Theta (y) - x \Vert ^2 $ is differentiable and attains its minimum at a point  $t = (t_1, \ldots , t_n)$ in the interior $(-r,r) ^{\times n}$. By the chain rule, 
    \[      \tfrac{\partial f }{\partial t_j } (t) =   (  \Theta (t)  - x  ) \cdot \tfrac{\partial }{\partial t_j} \Theta (t)  = 0       \]
    for each $j \in \{ 1, \ldots , n \} $. This implies that either $ \Theta (t) = x  $ or the set of partial derivatives $\{ \tfrac{\partial }{\partial t_j} \Theta (t)  \} _{j = 1 } ^n $ is linearly dependent. However, by \eqref{eq:hypothesis-good-partials} the latter possibility is ruled out.

    To show  uniqueness of $t$, assume that for a distinct $t' \in (-r,r)^{\times n} $ one also has $\Theta (t') = x$. By Rolle's Theorem there would be points $y_1, \ldots , y_n  \in  U $ in the segment between $t$ and $t'$ such that
    \[   (t- t') \cdot \left[ \nabla  \Theta _j  (y_j) \right] = 0    \]
    for each $j \in \{ 1, \ldots, n \} $, where $\Theta _j$ denotes the $j$-th component function of $\Theta$. This would contradict the linear independence of the set $\{ \nabla  \Theta _j (y_j)  \} _{j = 1} ^n$ provided by \eqref{eq:hypothesis-good-partials}. 
\end{proof}

\begin{proof}[Proof of Lemma \ref{lem:lie-ift}] We build upon the proof of  Lemma \ref{lem:bch-quant}. Recall that the exponential map is a local diffeomorphism in $B : = B_{1/2C}^{\mathfrak{g}}(0)$. Let $V_1, \ldots , V_n \in \mathfrak{X}(G) $ be the left-invariant vector fields defined by 
\[    V_j (g) : = (L_g)_{\ast} (\vv _ j ) ,               \] 
and define their lifts $\tilde{V}_1, \ldots , \tilde{V}_n \in \mathfrak{X}(B)$ by the formula  
\[ 
    \tilde{V} _j (\xx) : = \Psi_{\xx} ^{-1} (\vv _j ) = [ d_{\xx} \exp ] ^{-1} V_j (\exp (\xx))  ,
\] 
where $\Psi _{\xx} :  T_{\xx}\mathfrak{g} \to \mathfrak{g}$ is given by \eqref{eq:psi-x-definition}.  

If $ \Vert \xx \Vert \leq \frac{\eta }{2C}$ for some  $ \eta \in ( 0 , 1 ) $, by \eqref{eq:psi-inverse-good} we have
\begin{equation}\label{eq:v-tilde-and-v}
      \Vert    \tilde{V}_j (\xx) - \vv _j \Vert \leq  \eta .  
\end{equation}
In particular,
\begin{equation}\label{eq:v-tilde-bound}
\Vert \tilde{V}_j (\xx) \Vert \leq 2
\end{equation} for all $\xx \in B$.  We denote by  $ \Phi ^j : G  \times \mathbb{R}  \to G $ and $\tilde{\Phi} ^j : B_{1/4C}^{\mathfrak{g}}(0) \times (- \frac{1}{8C}, \frac{1}{8C}) \to B$  the flows of $V_j$ and $\tilde{V}_j$, respectively. Note that 
\begin{equation}\label{eq:phi-formula}
\Phi^j_t (g)  = g \exp (t\vv_j)     
\end{equation}
for all $g \in G$ and $t \in \mathbb{R}$. Since the vector fields $V_j$ and $\tilde{V}_j$ are related by $\exp$, for $\yy \in B_{1/ 4C}^{\mathfrak{g}}(0) $ and $t \in  (- \frac{1}{8C}, \frac{1}{8C})$ we have
\begin{equation}\label{eq:phi-exp-related}
 \exp ( \tilde{\Phi} ^ j _t (\yy)) =  \Phi ^j _t ( \exp (\yy) )    .
\end{equation}
 By taking the derivative,  we obtain
\begin{equation*}\label{eq:phi-tilde-exp}
    [ d_{\tilde{\Phi}^j _t (\yy) }  \exp ] \circ (\tilde{\Phi}_t^j)_{\ast}   =  ( \Phi ^j _t )_{\ast} \circ  [d_{\yy}\exp ] ,
\end{equation*}
and rearranging,
\begin{equation*}\label{eq:phi-tilde-exp}
    (\tilde{\Phi}_t^j)_{\ast} \circ [d_{\yy}\exp ] ^{-1} = [ d_{\tilde{\Phi}^j _t (\yy) }  \exp ]^{-1} \circ ( \Phi ^j _t )_{\ast}  .
\end{equation*}
Using this for $\yy  = \tilde{\Phi} _{-t}^j (\xx)$, we compute
\begin{equation}\label{eq:flow-of-left-invariant}
\begin{split}
  (  \tilde{\Phi} ^j _t )_{\ast} \circ \Psi ^{-1}_{\tilde{\Phi}^j_{-t}(\xx)} & = (\tilde{\Phi}^j_t) _{\ast} \circ [ d_{\tilde{\Phi}^j _{-t}(\xx)}\exp ]^{-1} \circ (L_{ \exp (\tilde{\Phi}^j_{-t} (\xx)) } )_{\ast} \\
    & = [  d_{\xx} \exp ]^{-1} \circ (\Phi ^j _t)_{\ast} \circ (L_{\exp (\tilde{\Phi}^j_{-t} (\xx) ) } )_{\ast} \\
    & = [  d_{\xx} \exp ]^{-1} \circ  (  R_{\exp ( t \vv _j )} )_{ \ast}  \circ (L_{\exp ( \xx ) } )_{\ast} \circ  (  L_{\exp (- t \vv _j )} )_{ \ast}   \\
    & = [  d_{\xx} \exp ]^{-1}  \circ (L_{\exp ( \xx ) } )_{\ast} \circ  (  L_{\exp (- t \vv _j )} )_{ \ast} \circ  (  R_{\exp ( t \vv _j )} )_{ \ast}   \\
    & = \Psi _{\xx}^{-1} \circ \Ad _{\exp (-t\vv_j)} \\
    & = \Psi _{\xx}^{-1} \circ  e ^{ - t \cdot \ad _{\vv_j} } ,  
\end{split}
\end{equation}
where the third line follows from  \eqref{eq:phi-formula} and \eqref{eq:phi-exp-related}. Now we define 
\[  \Theta  : ( - \tfrac{1}{8nC} , \tfrac{1}{8nC} ) ^{\times n} \to  B   \]
by 
\[     \Theta (t_1, \ldots , t_n ) : = \tilde{\Phi} ^n _{t_n} \circ \cdots \circ \tilde{\Phi} ^1 _{t_1} (0) .        \]
From \eqref{eq:v-tilde-bound}, we have
\begin{equation}\label{eq:theta-bound}
    \Vert \Theta ( t_ 1, \ldots , t_n ) \Vert \leq 2 \sum_{i = 1} ^n \vert t_i \vert 
\end{equation}
for all $(t_1, \ldots , t_n ) \in  ( - \tfrac{1}{8nC} , \tfrac{1}{8nC} ) ^{\times n} $. We now compute the partial derivatives of $\Theta$: 
\begin{align*}
    \tfrac{\partial}{\partial t_j} \Theta (t_1, \ldots , t_n) & = \tfrac{\partial}{\partial t_j}   ( \tilde{ \Phi } ^n _{t_n} \circ \cdots \circ \tilde{ \Phi } ^1 _{t_1} (0) ) \\
    & =  (\tilde{\Phi}_{t_n} ^n)_{\ast} \cdots (\tilde{\Phi}_{t_{j+1}} ^{j+1})_{\ast} \left[   \tilde{V}_j (  \tilde{\Phi} ^j_{t_j} \circ \cdots \circ \tilde{\Phi}^1_{t_1} (0)  )\right] . 
\end{align*}
If we set $ \xx  : = \tilde{\Phi}^j_{t_j} \circ \cdots \circ \tilde{\Phi}^1_{t_1} (0) , $ then by repeated applications of \eqref{eq:flow-of-left-invariant} we get
\begin{align*}
     \tfrac{\partial}{\partial t_j} \Theta (t_1, \ldots , t_n)   & =   (\tilde{\Phi}_{t_n} ^n)_{\ast} \cdots (\tilde{\Phi}_{t_{j+1}} ^{j+1})_{\ast}  \Psi_{ \xx } ^{-1}  ( \vv _j )    \\
  &  = \Psi _{  \tilde{\Phi }^n _{t_n} \cdots \tilde{\Phi}^{j+1} _{t_{j+1}} (\xx) } ^{-1} \left(  e ^{-t_n\cdot  \ad_{\vv_n} } \cdots e ^{-t_{j+1} \cdot  \ad_{\vv_{j+1}} } ( \vv_j ) \right) \\
  & = \Psi _{ \Theta (t_1, \ldots , t_n ) } ^{-1} \left(  e ^{-t_n\cdot  \ad_{\vv_n} } \cdots e ^{-t_{j+1} \cdot  \ad_{\vv_{j+1}} } ( \vv_j ) \right)  
\end{align*}
If for some $\alpha \in ( 0 , 1 ) $ we have  $\vert t_i \vert \leq \tfrac{\alpha }{ 8 n C  }   $  for each $i$, then
\begin{equation}\label{eq:good-partials-1}
    \begin{split}
    \Vert \tfrac{\partial}{\partial t_j} \Theta & (t_1, \ldots , t_n) - \tilde{V}_j ( \Theta (t_1, \ldots , t_n)  ) \Vert \\ 
    & =  \Vert     \Psi_{\Theta (t_1, \ldots, t_n ) } ^{-1} \left(  e ^{-t_n \cdot \ad _{\vv_n}}  \cdots e ^{- t_{j + 1} \cdot \ad _{\vv _ {j+1}}}  (\vv _j) \right) -   \Psi_{\Theta (t_1, \ldots, t_n ) } ^{-1} ( \vv_j )  \Vert  \\
    & \leq \Vert   \Psi ^{-1} _{\Theta (t_1, \ldots , t_n )} \Vert \Vert e ^{-t_n \cdot \ad _{\vv_n}}  \cdots e ^{- t_{j + 1} \cdot \ad _{\vv _ {j+1}}} - \Id _{\mathfrak{g}}  \Vert \\
    & \leq  \Big[ 1 +  2C \Vert \Theta (t_1, \ldots , t_n ) \Vert   \Big] \Big[  2 \tfrac{\alpha}{8}  e ^{\frac{\alpha }{8}}    \Big]   \\
    & \leq \tfrac{\alpha}{2} ,
    \end{split}
\end{equation}
where in the fourth line we used \eqref{eq:exponential-bound} and \eqref{eq:psi-inverse-good}, and in the last line we used \eqref{eq:theta-bound}.  Under these conditions, from \eqref{eq:v-tilde-and-v} and \eqref{eq:theta-bound} we also have
\begin{equation}\label{eq:good-partials-2}
\Vert \tilde{V}_j ( \Theta (t_1, \ldots ,  t_n)  ) - \vv _j  \Vert   \leq \tfrac{\alpha}{2} .
\end{equation}
Therefore, combining \eqref{eq:good-partials-1} and \eqref{eq:good-partials-2}, we get 
\[    \Vert \tfrac{\partial}{\partial t_j} \Theta  (t_1, \ldots , t_n)  - \vv _j  \Vert    \leq \alpha .                                      \]
If we set $r \leq \min \{ \varepsilon , \frac{1}{16 n ^2 C } \} $, for any $t_1, \ldots , t_n \in [ - r, r ] $ we get  
\[      \Vert \tfrac{\partial}{\partial t_j} \Theta  (t_1, \ldots , t_n)  - \vv _j  \Vert    \leq \tfrac{1}{2n}   ,    \]
so by Theorem \ref{thm:inverse-function-theorem}, for any $\ww \in B_{r/4} (0)$ there are $t_1, \ldots , t_n \in ( - r,r ) $ with 
\[      \ww = \Theta (t_1, \ldots , t_n ) .     \]
By multiple applications of \eqref{eq:phi-exp-related}, we conclude
\begin{align*}
    \exp (\ww) & = \exp ( \Theta ( t_1, \ldots , t_n ) ) \\
    & = \exp (  \tilde{\Phi } ^n _{t_n} \circ \cdots \circ \tilde{\Phi}^1 _{t_1} (0)   ) \\
    & = \Phi ^n_{t_n} \circ \cdots \circ \Phi ^1_{t_1} (\exp (0)) \\
    & = \exp (t_1 \vv_1) \cdots \exp (t_n \vv _n)
\end{align*}   
\end{proof}

\section{Local good approximations}\label{sec:local-good}

As Example \ref{ex:bad-approximation} shows, good approximations do not behave well with respect to quotients. To get around this issue, we adapt the notion of good approximations  to maps between local groups.

\begin{Def}[Local good approximations] Let $G_i$ be a sequence of locally compact Hausdorff local groups and $G$ a locally compact Hausdorff local group. We say a sequence of functions $\phi _ i : G_i \to G$ consists of \emph{local good approximations} if there are open, pre-compact, symmetric, multiplicative sets $A_i \subset G_i$, $A \subset G$ such that:
    \begin{enumerate}[label=\Roman*${}_{\text{loc}}$]
        \item (Almost surjectivity) For all $U \subset A $ open nonempty, one has $\phi_i (A_i) \cap U \neq \emptyset$ for $i$ large enough.\label{item:lga-1}
        \item (No expansion) For all $V \subset G$ open with $\overline{A} \subset V$, one has $\phi_i (A_i) \subset V$ for $i$ large enough.\label{item:lga-2}
        \item (Almost homomorphism) For each identity neighborhood $U \subset G$, there is $i_0 \in \mathbb{N}$ such that if $i \geq i_0 $ and   $g, h \in A_i ^8 $, then $ [ \phi_i (gh)^{-1} \phi _i (g) \phi_i (h) ]  \in U $. \label{item:lga-3}
        \item (No compression) For each compact $K\subset A$, one has 
        \[ \phi_i ^{-1} (K) \cap A_i^8 \subset A_i  \] 
        for $i$ large enough. \label{item:lga-4} 
        \item (Almost continuity) For each identity neighborhood $U \subset G$, there is a sequence of identity neighborhoods $U_i \subset G_i$ with $\phi_i (U_i) \subset U $ for $i$ large enough. \label{item:lga-5}
    \end{enumerate}
    If properties (I${}_{\text{loc}}$--V${}_{\text{loc}}$) hold, we call the sets $A_i$ and $ A$ \emph{regular neighborhoods} with respect to the approximations $\phi_i$. 

    \emph{Small subgroups} and the \emph{NSS property} in the context of local good approximations are defined in exactly the same way as they were defined for good approximations (see Definition \ref{def:ss2}). 
\end{Def}

The notion of good approximations was originally inspired by Definition \ref{def:good-model} below.  Good models of ultraproducts of discrete multiplicative sets were introduced in \cite{hrushovski, breuillard-green-tao}, and extended in \cite{carolino} to locally compact groups.  Recall that throughout this paper, we are working with a fixed non-principal ultrafilter $\alpha$. 

\begin{Def}[Good model]\label{def:good-model}
    Let $G_i$ be a sequence of local groups, $A_i \subset G_i$ a sequence of symmetric, open, pre-compact multiplicative sets, and $L$ a local group.  We say a morphism  $\pi : \prod _{i \to \alpha } A_i ^8 \to L$ is a \emph{good model} if the following holds:
    \begin{enumerate}[label=\Roman*${}_{\text{model}}$]
        \item (Thick image) There is an open identity neighborhood $U_0 \subset L $ such that $U_0 \subset \pi (\prod _{i \to \alpha } A_i )$ and  $\pi ^{-1} (U_0) \subset \prod _{i \to \alpha } A_i $.\label{item:gm-1}
        \item (Compact image) $\pi (  \prod_{i \to \alpha } A_i  )$ is contained in a compact set.\label{item:gm-2}
        \item (Approximation by internal open sets) For any $K \subset U \subset U_0$ with $K$ compact and $U$ open,  there is a sequence of open sets $A_i ' \subset A_i$ such that 
        \[    \pi ^{-1} (K) \subset \prod_{i \to \alpha } A_i ' \subset \pi ^{-1} (U) .\label{item:gm-3}                  \]
    \end{enumerate}
\end{Def}

It is immediate from the definition that a sequence of good approximations is also a sequence of local good approximations with respect to the same regular neighborhoods. In the same direction, the following proposition shows that ultralimits of local good approximations are good models. 

\begin{Pro}\label{pro:lga-to-gm}
    Let $\phi_i : G_i \to G$ be a sequence of local good approximations with regular neighborhoods $A_i \subset G_i$ and $A \subset G$. Define $\pi : \prod_{i \to \alpha } A_i^8 \to G$ as 
    \begin{equation}\label{eq:gm-definition}
     \pi (  (  a_i )   ) : = \lim_{i \to \alpha } \phi_i (a_i) .     
    \end{equation}
    If $G$ is metrizable, then $\pi$ is a good model. 
\end{Pro}

\begin{proof}
    By \eqref{item:lga-2}, we have $\phi_i (A_i) \subset A^{2}$ for $i$ large enough. By \eqref{item:lga-3}, this implies $\phi_i (A_i^8) \subset A^{17}$ for $i$ large enough.  Consequently,  the ultralimit \eqref{eq:gm-definition} is well defined, and $\pi ( \prod _{i \to \alpha } A_i^8 ) \subset \overline{A}^{17}$, proving \eqref{item:gm-2}.     Again by \eqref{item:lga-3}, the map $\pi $ is a  morphism.

     Let $U_0 \subset A$ be an open identity neighborhood with $\overline{U_0} \subset A$.     To prove \eqref{item:gm-1}, fix $g \in U_0$ and $U \subset U_0$ a neighborhood of $g$. By \eqref{item:lga-1}, there are $g_i \in A_i$ with $\phi_i (g_i) \in U$ for $i$ large enough. Using a countable neighborhood basis of $g$ and a diagonal argument, we can find a sequence $a_i \in A_i$ with $\phi_i (a_i) \to g$. This implies 
    \[  \pi ((a_i)) = \lim_{i \to \alpha} \phi_i (a_i) = \lim _{i \to \infty} \phi_i (a_i) = g .   \]
    On the other hand, take $(a_i) \in \prod_{i \to \alpha} A_i ^8$ with $\pi ((a_i)) = g \in U_0 $. By definition of ultraconvergence, this implies $\alpha (\{ i \in \mathbb{N} \, \vert \, \phi_i (a_i) \in U_0 \} ) = 1$. Since $\overline{U_0} \subset A $ is compact,  \eqref{item:lga-4} implies $\alpha ( \{ i \in \mathbb{N} \, \vert \, a_i \in A_i  \} ) = 1$. Therefore,  $(a_i) \in \prod_{i \to \alpha} A_i$, proving \eqref{item:gm-1}.

    To prove \eqref{item:gm-3}, fix $K \subset U \subset U_0$ with $K$ compact and $U$ open. Let $K_1 \subset U_1 \subset U$ be such that  $K_1$ is compact,  $U_1$ is open, and 
    \begin{equation}\label{eq:long-chain}
        K \subset \inte (K_1 ) \subset K_1   \subset U_1 \subset \overline{U_1} \subset U .    
    \end{equation}
    By Lemma \ref{lem:approximation-by-internal-sets}, there is a sequence of open subsets $A_i' \subset A_i$ with 
    \begin{equation}\label{eq:approximation-by-internal-1}
        \phi_i ^{-1} (K _1 ) \cap A_i \subset A_i ' \subset \phi_i ^{-1} (U_1)          
    \end{equation}
    for $i$ large enough. We note that while Lemma \ref{lem:approximation-by-internal-sets} was proven for good approximations, the exact same proof works for  local good approximations. 

    On one hand, for any $(a_i) \in  \prod _{i \to \alpha } A_i ' $, by \eqref{eq:long-chain} and \eqref{eq:approximation-by-internal-1} we have 
    \[   \pi ((a_i))  =   \lim_{i \to \alpha} \phi_i (a_i) \in \overline{ U_1 } \subset U .          \]
    On the other hand, take $(b_i) \in \prod _{i \to \alpha } A_i ^8$ with $\pi ((b_i)) = g \in K$. By \eqref{eq:long-chain}, this implies
    \begin{equation}\label{eq:k-1}
        \alpha (\{ i \in \mathbb{N} \, \vert \,  \phi_i (b_i) \in \inte (K_1) \} ) = 1 . 
    \end{equation}
     By \eqref{item:lga-4}, we then have
     \begin{equation}\label{eq:a-1-1}
    \alpha (\{ i \in \mathbb{N} \, \vert \, b_i \in A_i \}  ) = 1.      
     \end{equation}
    Putting together  \eqref{eq:approximation-by-internal-1}, \eqref{eq:k-1}, and \eqref{eq:a-1-1}, we deduce  
    \[   \alpha ( \{ i \in \mathbb{N} \, \vert \, b_i \in A_i ' \}  )  = 1.           \] 
    Therefore, $(b_i) \in \prod _{i \to \alpha} A_i ' $, and we conclude
    \[     \pi ^{-1} (K ) \subset \prod_{i \to \alpha} A_i ' \subset \pi ^{-1} (  U  ) ,                    \]
    finishing the proof of \eqref{item:gm-3}.  
    \end{proof}

The following two theorems extend results from good approximations to local good approximations.

\begin{Thm}\label{thm:nsz-local}
    Let $G_i$ be a sequence of local Lie groups and  $\phi _i : G_i \to G$ a sequence of local good approximations. If $G$ is a local Lie group and  the sequence $G_i$ has the NSS property, then
    \[    \text{dim} (G) \geq \limsup_{i \to \infty} \text{dim} (G_i) .    \] 
\end{Thm}

\begin{proof}
    Each step but one in the proof of Theorem \ref{thm:nsz} in \cite{nsz} only uses the local good approximations rather than global good approximations. The only exception is Theorem \ref{thm:gleason}, since its proof uses the Haar measure on the groups $G_i$, whose construction is more delicate for  local groups than for globally defined  groups.  
    
    However, in the setting of local Lie groups, the Haar measure can be easily constructed as the Riemannian volume measure associated to a left-invariant Riemannian metric, so a version of Theorem \ref{thm:gleason} is available for local good approximations, as long as the local groups $G_i$ are Lie. 

    Consequently, the proof of Theorem \ref{thm:nsz} carries over to the setting of local Lie groups. 
\end{proof}

\begin{Thm}\label{thm:bgt-local} 
    Let $Q_i$ be a sequence of discrete local groups and  $\psi _i : Q_i \to Q$ a sequence of local good approximations  with $Q$ a local Lie group. Then the Lie algebra of $Q$ is nilpotent. 
\end{Thm}

Theorem \ref{thm:bgt-local} is a direct consequence of Proposition \ref{pro:lga-to-gm} and the following result.

\begin{Pro}\cite{breuillard-green-tao}. \label{pro:bgt}
    Let $G_i$ be a sequence of discrete local groups, $A_i \subset G_i$ a sequence of finite symmetric multiplicative sets, and   $\pi : \prod _{i \to \alpha } A_i ^8 \to L$ a good model. Then the Lie algebra of $L$ is nilpotent. 
\end{Pro}

\begin{proof}[Proof of Theorem \ref{thm:bgt-local}]
By Proposition \ref{pro:lga-to-gm}, the  ultralimit of the sequence $\psi _i$ produces a good model. Then the result follows from Proposition \ref{pro:bgt}. 
\end{proof}

\subsection*{Local good approximations between quotients} In the remainder of this section, we show that under the right conditions, local good approximations pass to local good approximations between local group quotients. For that goal, we require the following definition.

\begin{Def}[Approximation by subsets]
    Let $\phi_i : G_i \to G$ be a sequence of local good approximations with regular neighborhoods $A_i \subset G_i$ and $A \subset G$. We say a sequence of open subsets $B_i \subset A_i $ \emph{approximates} an open set $B \subset A$ if the following holds: 
    \begin{itemize}
        \item  For all $U \subset B $ open nonempty, one has $\phi_i (B_i) \cap U \neq \emptyset$ for $i$ large enough. 
        \item  For all $V \subset G$ open with $\overline{B} \subset V$, one has $\phi_i (B_i) \subset V$ for $i$ large enough.
        \item For each compact $K\subset B$, one has 
        \[   \phi_i ^{-1} (K) \cap A_i ^8  \subset B_i       \] 
        for $i$ large enough. 
    \end{itemize}
    In particular, if $G$ is metrizable,  $\phi_i (B_i)$ converges in the Hausdorff topology to $B$. 
\end{Def}

Combining Lemma \ref{lem:approximation-by-internal-sets} with a diagonal argument, one can approximate any open set (see the proof of \cite[Lemma 4.15]{nsz}).

\begin{Lem}\label{lem:approximate}
    Let $\phi_i : G_i \to G$ be a sequence of local good approximations with regular neighborhoods $A_i \subset G_i$ and $A \subset G$. If $G$ is metrizable, then for any open set $B \subset A$ there is a sequence of open sets  $B_i \subset A_i$ that approximates $B$. Moreover, if $B$ is symmetric, the sets $B_i$ can be taken to be symmetric.
\end{Lem}

Let $\phi_i : G_i \to G$ be a sequence of local good approximations with regular neighborhoods $A_i \subset G_i$ and $A \subset G$. Let $H_i \subset G_i$ and $H \subset G$  normal sub-local groups with normalizing neighborhoods $V_i \subset A_i$ and $V \subset A$, respectively.  Let $W_i \subset G _ i $ and $W \subset G$ be open symmetric sets with $W _i ^{6 } \subset  V_i $ and $W^{6} \subset V$, and denote  $(G_i/H_i)_{W_i}$ and $(G/H)_W$ by $Q_i$ and $Q$, respectively.

Assume: 
\begin{itemize}
    \item $G$ is metrizable.
    \item The sequence $\phi_i (H_i )$ converges in the Hausdorff topology to $H$.
    \item The sets $W_i$ approximate $W$.
\end{itemize}
Let $\pi _i :  W_i \to Q_i $ and $\pi : W \to Q$ denote the projections, equip $G$ with a compatible metric $d: G \times G \to \mathbb{R}$, and let $\delta _i \to 0 $ be such that the Hausdorff distance between $\phi_i (W_i) $ and $W$ is less than $\delta_i$. 

For  each $i$, let $\theta_i : Q_i \to W_i $ be a section, and define $\psi_i : Q_i \to Q$ as  $\psi_i (q) = \pi (x)$ where $x \in W$ satisfies  
    \[   d ( x , \phi_i (\theta_i (q))) < \delta _i  .   \]
Let $B \subset A$ be an open symmetric set with $\overline{B}^{200} \subset W $, and let $B_i \subset A_i$ be a sequence of open symmetric sets that approximate $B$.  
   
\begin{Pro}\label{pro:lga-quotient}
    The maps $\psi _i : Q_i \to Q $ are local good approximations with regular neighborhoods $\pi _ i (B_i ) \subset Q_i$ and $\pi ( B )  \subset Q$. 
\end{Pro}

Throughout the proof of Proposition \ref{pro:lga-quotient}, we make repeated use of the following fact.

\begin{Lem}\label{lem:quotient-compactness}
 Let  $k \leq 200 $ and $g_i \in B^{k}_i$.  If $\phi_i (g_i) \to g \in G$, then $g \in \overline{B}^k$, $g_i \in W_i$ for $i$ large enough, and  
 \begin{equation}\label{eq:psi-is-good}
     \psi_i ( \pi _i (g_i)) \to \pi (g)  .         
 \end{equation}
\end{Lem}
\begin{proof}
Since the sequence $W_i $ approximates $W$ and $\overline{B}^k \subset W$, the fact that $g_i \in W_i$ for $i$ large enough follows from $g \in \overline{B}^k$, which we now prove. For each $i$, write $g_i = b_{i,1} \cdots b_{i,k}$ with $b_{i,j} \in B_i$ for each $j$. After passing to a subsequence, for each $j$ we have $\phi_i (b_{i,j}) \xrightarrow[i \to \infty]{} b_j$ for some $b_j \in \overline{B} $.  This implies 
\[     \phi_i (g_i) \to b_1 \cdots b_k \in \overline{B}^k.       \]
This proves the first two claims, and  \eqref{eq:psi-is-good} remains to be proven. If it fails, after passing to a subsequence, no subsequence of $\psi _i (\pi _i (g_i))$ converges to $\pi (g)$.    By definition,  $\theta_i (\pi _i (g_i)) = g_i h_i \in W_i $ for some $h_i \in H_i \cap W_i ^2 $.  After passing to a subsequence, $\phi_i (h_i) \to h$ for some $h \in \overline{H} \cap \overline{W}^2 \subset H$.  Consequently, 
    \[    \phi_i ( \theta _i (\pi _i (g_i)))  \to g h  .      \]
    Then, $\psi _i (\pi _i (g_i)) = \pi (x_i)$ for a sequence $x_i \to gh$. By continuity of $\pi$, we conclude $\psi_i ( \pi _i (g_i)) \to \pi (g) $; a contradiction.  
\end{proof}

\begin{proof}[Proof of Proposition \ref{pro:lga-quotient}]
By Lemma \ref{lem:quotient-compactness}, for $i$ large enough,  $\pi_i (B_i) \subset Q_i $ is a multiplicative set.

To check \eqref{item:lga-1}, fix $q \in \pi (B)$ and $b \in B$ with $\pi (b) = q$. Since the sets $B_i$ approximate $B$, there are $b_i \in B_i$ with $\phi_i (b_i) \to b$. By Lemma \ref{lem:quotient-compactness}, we have  $\psi_i (\pi _i (b_i)) \to q$.   Since $q$ was arbitrary, \eqref{item:lga-1} is proved.

Assuming \eqref{item:lga-2} fails, after passing to a subsequence, there are $b_i \in B_i$ with $\psi_i (\pi _i (b_i)) \to q \in Q \backslash \overline{\pi (B) } $. After further passing to a subsequence, $\phi_i (b_i) \to b$ for some $b \in \overline{B}$.  By Lemma \ref{lem:quotient-compactness}, we have $ q =  \pi (b) \in \pi (\overline{B} ) \subset \overline{\pi (B)}$; a contradiction.

Assuming \eqref{item:lga-3} fails, after passing to a subsequence, there is an identity neighborhood $U \subset Q$, and $a_i , b_i \in B_i ^{8}$ with
\begin{equation}\label{eq:bad-homomorphism}
     \psi_i (\pi_i(a_ib_i))^{-1} \psi_i (\pi_i(a_i)) \psi _i (\pi_i (b_i)) \notin U     
\end{equation}
for all $i$. After further passing to a subsequence, we have   
\[   
    \phi_i (a_i)  \to a, \hspace{2cm}
    \phi_i (b_i)  \to b,  \hspace{2cm}
    \phi_i (a_ib_i)  \to ab,                 \]
for some $a,b \in \overline{B}^8$.  By Lemma \ref{lem:quotient-compactness}, we have
\[   
    \psi_i (\pi_i(a_i))  \to \pi (a), \hspace{1.5cm}
    \psi_i (\pi_i (b_i))  \to \pi (b),  \hspace{1.5cm}
    \psi_i (\pi_i (a_ib_i) )  \to \pi(ab),             \]
    contradicting \eqref{eq:bad-homomorphism}.

Assuming \eqref{item:lga-4} fails, after passing to a subsequence,  there are $b_i \in B_i^8 $  with $\pi _i (b_i) \notin \pi (B_i)$ for all $i$ but $\psi_i (\pi _i (b_i)) \to q$ for some $q  \in \pi (B)$.  After further passing to a subsequence, we have  $ \phi_i (b_i) \to a $ for some $a  \in \overline{B}^8$. By Lemma \ref{lem:quotient-compactness}, we have $\pi (a) = q \in \pi (B)$, so $a = b h$ with $b \in B$ and $h \in H \cap W  $.  

Pick $h_i \in H_i$ with $\phi_i (h_i) \to h$.  Then $\phi_i ( b_i h_i^{-1} ) \to b \in B$, so $b_i h_i ^{-1} \in B_i$ for $i$ large enough. This implies $\pi _i (b_i) \in \pi (B_i)$; a contradiction.

Finally, to prove \eqref{item:lga-5}, let $U_0 \subset Q$ be an open identity neighborhood. Since $\pi$ is continuous, there is an identity neighborhood $V_0 \subset W$ with $\overline{V_0} \subset \pi ^{-1} (U_0)$. Let $V_i \subset B_i$ be a sequence of identity neighborhoods with $\phi_i (V_i ) \subset V_0$ for $i$ large enough. We claim that $\psi _i (\pi_i (V_i)) \subset U_0$ for $i$ large enough. If this is not the case, after passing to a subsequence, there are $b_i \in V_i$ with $\psi_i (\pi _i (b_i)) \to q $ for some $q \notin  U_0 $. After passing to a subsequence, we also have $\phi_i (b_i) \to b$ for some $b \in \overline{V}_0 $. By  Lemma \ref{lem:quotient-compactness}, we conclude  $ q  = \pi (b) \in \pi (\overline{V_0}) \subset U_0$; a contradiction.    
\end{proof}

\section{Reduction to the NSS case}\label{sec:to-nss}

The goal of this section is to reduce the proof of Theorem \ref{thm:main-good} to the case when the sequence $G_i$ has the NSS property. This case is addressed in the theorem below.  

\begin{Thm}\label{thm:main-nss}
     Let $G_i$ be a sequence of Lie groups and  $\phi _i : G_i \to G$ a sequence of good approximations. Assume $G$ is a Lie group and the sequence $G_i$ has the NSS property. Then $\mathfrak{g}$, the  Lie algebra of $G$, admits an ideal $\mathfrak{h} \trianglelefteq \mathfrak{g}$ with $\mathfrak{g}/ \mathfrak{h}$ nilpotent and 
    \begin{equation}\label{eq:h-ineq-3}
         \text{dim} (\mathfrak{h}) = \limsup_{i \to \infty} \text{dim} (G_i) .  
    \end{equation}
\end{Thm}

To reduce Theorem \ref{thm:main-good} to Theorem \ref{thm:main-nss}, we need an easy technical lemma. It states that in the presence of small subgroups $H_i \trianglelefteq G_i$, one can quotient by $H_i$ and retain the good approximation.

\begin{Lem}\label{lem:small-quotient}
Let $\phi_i : G_i \to G$ be a sequence of good approximations with regular neighborhoods $A_i \subset G_i$, $A \subset G$. Consider a sequence of small subgroups  $H_i \leq G_i$  normalized by $A_i$ for $i$ large enough and let $G_i' : = \langle A_i \rangle  / H_i$. If we denote by $\pi _i : \langle A_i \rangle \to G_i'$ the natural projection, then
\begin{itemize}
    \item For any sequence of sections $\theta_i : G_i ' \to \langle A_i \rangle $, the maps $\phi_i ' : G_i ' \to G$ given by $\phi_i ' : = \phi_i \circ \theta_i$ are good approximations with regular neighborhoods $A_i ' : = \pi _i (A_i)  \subset G_i'$ and $A \subset G$.
    \item Given a sequence of subgroups $W_i \leq G_i ' $ with $W_i \subset A_i ' $ for all $i$ large, the sequence $W_i$ consists of small subgroups if and only if the sequence of preimages $\pi ^{-1} _i ( W_i) \leq G_i$ consists of small subgroups.  
\end{itemize}
\end{Lem}

\begin{proof}
For $g \in \langle A_i \rangle$, we have $\theta_i (\pi_i (g)) = gh$ for some $h \in H_i$. Then
\[ 
   \phi_i' ( \pi_i (g) ) = [  \phi _i (g)  ] \ast [ \phi_i (h )  ] \ast  [  \phi_i(h )^{-1} \phi_i (g ) ^{-1} \phi _i (gh )  ]  .
\] 
This means that for any  $n \in \mathbb{N}$ and any identity neighborhood $U \subset G$,  there is $i_0 \in \mathbb{N}$ such that  for $i\geq i_0$,  any $g \in A_i ^n$ satisfies 
\[   \phi_i (g) ^{-1}  \phi_i ' ( \pi _i (g)  )    \in U  .   \]
Using this, all the properties of good approximations for the pairs $(\phi_i' , A_i')$ can be checked directly from the corresponding ones for the pairs $(\phi_i , A_i)$.

From the above observation, a sequence $g_i \in  A_i ^2$ satisfies $\phi_i (g_i) \to e$ if and only if $\phi_i' (\pi_i (g_i)) \to e$. The second part of the lemma follows easily from this fact. 
\end{proof}

\begin{proof}[Proof of Theorem \ref{thm:main-good} assuming Theorem \ref{thm:main-nss}:]   
Let $G_i$, $G$, $\phi_i$ be as in Theorem \ref{thm:main-good}. Let $H_i \leq G_i$ be the groups given by Theorem \ref{thm:largest-small} and $\phi'_i : G_i' : = \langle A_i \rangle / H_i \to G $ the induced good approximations given by Lemma \ref{lem:small-quotient}. We claim that the sequence $G_i' $ has the NSS property. To see this, consider a sequence $W_i \leq G_i'$ of small subgroups. By  the second part of Lemma \ref{lem:small-quotient}, the preimages $\tilde{W}_i \leq G_i$ are small, hence contained in $H_i$ for $i$ large enough. This implies that the quotients $W_i = \tilde{W}_i / H_i$ are trivial for $i$ large enough, proving our claim.

By Theorem \ref{thm:main-nss} applied to the good approximations $\phi_i' : G_i' \to G$, there is an ideal $\mathfrak{h} \trianglelefteq \mathfrak{g}$ with $\mathfrak{g} / \mathfrak{h}$ nilpotent and 
\[            \text{dim} (\mathfrak{h}) = \limsup_{i \to \infty} \text{dim} (G_i')    \leq \limsup_{i \to \infty } \text{dim} (G_i) .   \]
Moreover, if the sequence $G_i$ has the NSS property, then $H_i$ is trivial for $i$ large enough, and 
\[                 \text{dim} (\mathfrak{h}) = \limsup_{i \to \infty} \text{dim} (G_i')  =   \limsup_{i \to \infty } \text{dim} (G_i) .  \]
\end{proof}

\section{Convergence of one-parameter subgroups}\label{sec:ops-convergence}

Sections \ref{sec:ops-convergence} to \ref{sec:properties-of-h} contain the proof of Theorem \ref{thm:main-nss}.  Throughout, we assume its hypotheses.  Let $B\subset G$, $B_i\subset G_i $, and $C_0 $ be as in Theorem \ref{thm:gleason}. Further assume  $B$ is so small that the exponential map is a diffeomorphism from 
\[     \{ \vv \in \mathfrak{g} \, \vert \, \vert \vv \vert _B \leq 10 \}       \]
onto its image in $G$. In particular, $\overline{B}$ does not contain non-trivial subgroups,  square roots in $\overline{B}$ are unique, and 
\[    \log : \overline{B} \to \{  \vv \in \mathfrak{g} \, \vert \, \vert \vv \vert _B \leq 1 \}         \]
is well defined.

In this section we study how one-parameter subgroups of the groups $G_i$ converge to one-parameter subgroups of $G$. We then use this to give a precise estimate between the escape norms $\Vert \cdot \Vert _{B_i}$ and $\vert \cdot \vert _{B_i}$. We denote by $\mathfrak{g}_i$ the Lie algebra of $G_i$.

\begin{Def}[Convergence of one-parameter subgroups]
Given a sequence $\vv _i \in \mathfrak{g}_i$ with $\vert \vv _i \vert _{B_i} \leq 1 $ for each $i$, we say the sequence of one-parameter subgroups  $\gamma_i : \mathbb{R} \to G_i $ given by  
\[            \gamma_i (t) : =   \exp (t \vv _i )     \]
\emph{converges} to a one-parameter subgroup $\gamma : \mathbb{R} \to G$ if for any $T > 0 $ and any identity neighborhood $ U \subset G $, for $i$ large enough one has 
\[     \phi_i (  \gamma _i (t)   ) ^{-1} \gamma (t)     \in U    \]
for all $t \in [-T,T]$. 
\end{Def}

The following result provides a compactness property for one-parameter subgroups.

\begin{Pro}\label{pro:convergence-of-ops}
Let $\vv _i \in \mathfrak{g}_i$ be a sequence with $\vert \vv _i \vert _{B_i} = 1 $ for each $i$. Then after passing to a subsequence, the sequence of one-parameter subgroups $\gamma_i : \mathbb{R} \to G_i$ given by
\[            \gamma_i (t) :=   \exp (t \vv _i )     \] 
converges to a one-parameter subgroup $\gamma: \mathbb{R} \to G$ of the form 
\[      \gamma (t)  = \exp (t \vv )                      \]
 with $ \vert \vv \vert _B = 1$. 
\end{Pro}

For the proof of Proposition \ref{pro:convergence-of-ops} we require a preliminary result regarding the equicontinuity of the compositions $\phi_i \circ \exp : \mathfrak{g}_i \to G$.

\begin{Lem}\label{lem:small-norm-is-good}
    For any identity neighborhood $ W \subset  G$, there is $\delta > 0 $ such that for $i$ large enough one has 
    \[     \phi _i ( \exp (\ww ) )  \in W         \]
    for all $\ww \in \mathfrak{g}_i $ with $\vert \ww \vert _{B_i} < \delta$. 
\end{Lem}

\begin{proof}
    By compactness, there is $m \in \mathbb{N}$ such that for any $g \in B \backslash \inte ( W ) $, one has $g ^k \notin \overline{B} $ for some $k \leq m$.  Define $\delta : = \frac{1}{m}$.  If the conclusion of the lemma fails, after passing to a subsequence, there is a sequence $\ww_i \in \mathfrak{g}_i $ with $\vert \ww _i \vert _{B_i} < \delta $ but  
    \[ \phi_i (\exp (\ww _i ) ) \notin W . \]
    After again passing to a subsequence, we can assume the sequence $\phi _i (\exp (\ww_i)) $ converges to some $g \in B \backslash \inte (W)$. By our choice of $m$, there is $k \leq m$ with $g ^ k \notin \overline{B}$, and 
    \[       \phi_i (\exp ( k \ww _i ) )   \xrightarrow[i \to \infty]{}   g^k .                                        \]
    For $i$ large enough, this would imply $\exp (k \ww _i ) \notin B_i$ and $\vert \ww _i \vert_{B_i} \geq \frac{1}{k} \geq \delta$, which is a contradiction. \end{proof}

        \begin{proof}[Proof of Proposition \ref{pro:convergence-of-ops}]  
Throughout this proof, we denote $\phi_i (\gamma _i (t))$ by $\alpha_i (t)$.   Note that $\gamma_i (1) \in \partial B_i$ for all $i$,  so by Lemma \ref{lem:convergent-sequences}, after passing to a subsequence, we have 
\[     \alpha _i (1 ) \longrightarrow   g_{\infty}                              \]
for some  $g _{\infty} \in  \partial B$.  Let $\vv : = \log (g_{\infty})$.  Since $B$ is the exponential of a small symmetric convex subset of $\mathfrak{g}$, we have $\vert \vv \vert _B = 1$.  Define $\gamma : \mathbb{R} \to G$ by
\[        \gamma (t) : = \exp (t \vv ).           \]
    Notice that any subsequence of $\alpha _i ( \tfrac{1}{2} ) $ has a subsequence that converges to a square root of $\gamma (1) $ in $B$, which is necessarily  $\gamma (\frac{1}{2})$. This implies 
    \[  \alpha _i \left( \tfrac{1}{2} \right)   \longrightarrow  \gamma \left( \tfrac{1}{2}\right)  .   \]
    Arguing inductively, the sequence $\alpha _i \left( \frac{1}{2^k} \right)  $ converges to $\gamma \left( \frac{1}{2^k} \right)$ for all $k \in \mathbb{N}$. This implies that for any dyadic rational $q  $ one has 
    \[  \alpha _i ( q )  \longrightarrow  \gamma (q)  .   \]    
    Assuming the proposition fails, after passing to a subsequence, there is $T > 0 $, an identity neighborhood $U \subset  G$,  and a sequence $t_i \in [-T,T]$ such that 
    \[       \alpha _i ( t_i ) ^{-1} \gamma (t_i) \notin U     \]
    for all $i$.     After further passing to a subsequence, we can assume $t_i \rightarrow t$ for some $t \in [-T,T]$. Let $W \subset G $ be a symmetric identity neighborhood  with $W^5 \subset U$.  By Lemma \ref{lem:small-norm-is-good}, there is $\delta > 0 $ such that for $i$ large enough, 
        \[ \alpha _i  ( s  )  \in W \] 
    for all $s \in [- \delta , \delta ]$. Pick a dyadic rational $q $ with $\gamma (q) ^{-1} \gamma (t ) \in W$ and $\vert t- q \vert < \delta$. Then 
    \begin{align*}
         \alpha_i (  t_i  ) ^{-1} \gamma (t_i) \, = \,  [ \alpha _i  (t_i )  ^{-1} & \alpha_i (q) \alpha_i  (t_i - q)   ] \,  \ast \, [  \alpha _i ( t_i - q  ) ^{-1} ]\, \ast \,  [  \alpha _i (q ) ^{-1} \gamma (q)   ]  \\
         & \ast \, [  \gamma (q) ^{-1}  \gamma (t)  ] \, \ast \, [\gamma (t) ^{-1} \gamma (t_i)]   .
    \end{align*}
    For $i$ large enough each term on the right belongs to $W$, showing that $ \alpha_i (  t_i  ) ^{-1} \gamma (t_i)  \in U $; a contradiction.    \end{proof}

With Proposition \ref{pro:convergence-of-ops}, we can prove a form of continuity for the Lie algebra escape norms. 

\color{black}

\begin{Lem}\label{lem:lie-norm-continuity}
    Let $\ww_i \in \mathfrak{g}_i$ be a sequence with $\vert \ww _i \vert _{B_i} \leq 1$ for all $i$. Assume $\phi_i (\exp (\ww _ i )) \to \exp (\ww)$ for some $\ww \in \mathfrak{g} $ with $\vert \ww \vert _B \leq 1$. Then 
    \[  \lim_{i \to \infty}  \vert \ww _i \vert _{B_i} =  \vert \ww \vert _B .            \]
\end{Lem}

\begin{proof}
    Arguing by contradiction, we can assume the sequence  $\vert \ww _i \vert _{B_i}$ converges to a  number  $s \neq \vert \ww \vert _B$.   For each $i$, let  $\vv_i : = \tau_{B_i} (\ww _i ) \ww _i $ and define $\gamma _i : \mathbb{R} \to G_i$ by  $ \gamma _i (t) : = \exp (t \vv_i ) . $   By Proposition \ref{pro:convergence-of-ops}, after further passing to a subsequence, the sequence $\gamma_i$ converges to a one-parameter subgroup $\gamma : \mathbb{R} \to G$  of the form  $   \gamma (t)   = \exp (t \vv )  $  with $\vert \vv \vert _B = 1$. Then 
    \[ \gamma (s) = \lim_{i \to \infty} \phi _i (\gamma _ i (\vert \ww_i \vert _{B_i}) ) =  \lim_{i \to \infty}  \phi_i (\exp (\ww _i ))  = \exp (\ww) .  \] 
    Hence $  \ww  = s \vv $; a contradiction. 
\end{proof}

We can also use Proposition \ref{pro:convergence-of-ops} to give a more precise connection between the escape norms $\Vert \cdot \Vert _{B_i}$ and $\vert \cdot \vert _{B_i}$.

    \begin{Lem}\label{lem:norms-are-equal-if-small}
    For any $\varepsilon > 0 $ there is  $\delta > 0 $ such that for $i$ large enough, any $\ww \in  B_{\delta} (0) \subset \mathfrak{g}_i$ satisfies
    \begin{equation}\label{eq:escape-equals-norm-if-small}
        (1 - \varepsilon ) \vert \ww \vert _{B_i} \leq \Vert \exp (\ww) \Vert _{B_i} \leq \vert \ww \vert _{B_i}.
    \end{equation}
    \end{Lem}
    \begin{proof} By \eqref{eq:escape-less-than-norm}, the second inequality always holds, so if \eqref{eq:escape-equals-norm-if-small} fails, after passing to a subsequence, there would be  $\ww _i \in \mathfrak{g}_i$ with $\vert \ww _i \vert _{B_i } \to 0$ and 
    \begin{equation}\label{eq:escape-not-equals-norm}
         \Vert \exp (\ww _i ) \Vert _{B_i} \leq (1 - \varepsilon) \vert \ww _i \vert _{B_i}.
    \end{equation}    
    For each $i$, let $\vv_i : =\tau _{B_i} (\ww_i) \ww _i $ and define  $\gamma _i : \mathbb{R} \to G_i$ by  $ \gamma _i (t) : = \exp (t \vv_i ) . $   By Proposition \ref{pro:convergence-of-ops}, after further passing to a subsequence, the sequence $\gamma_i$ converges to a one-parameter subgroup $\gamma : \mathbb{R} \to G$  of the form  $   \gamma (t)   = \exp (t \vv )  $  with $\vert \vv \vert _B = 1$.

    Since $\tau _{B_i} (\ww _i ) \to \infty$, we can choose $m_i\in \mathbb{N}$  so that $m_i / \tau_{B_i}  (\ww_i) \to 1 + \frac{\varepsilon}{2}$. Then by \eqref{eq:escape-not-equals-norm}, for $i$ large we have 
    \[             m_i  < (1 + \varepsilon) \tau _{B_i} (\ww_i)          \leq  (1 - \varepsilon) ^{-1}\vert \ww _i \vert _{B_i} ^{-1}   \leq  \Vert \exp (\ww _ i ) \Vert _{B_i} ^{-1} . 
        \]
    Hence  
    \[  \gamma _i (  m_i / \tau_{B_i}  (\ww_i) ) =   \exp (m_i \ww_i  ) = \exp (\ww _i ) ^{m_i}  \in B_i ,\] 
    but $\phi_i ( \gamma _i (  m_i / \tau _{B_i} (\ww_i) )  )$ converges to $\gamma ( 1 + \tfrac{\varepsilon}{2} ) \notin \overline{B}$. This contradicts \eqref{item:ga-2}. \end{proof}

\section{Construction of limit Lie algebra}\label{sec:h-construction}

We now resume the proof of  Theorem \ref{thm:main-nss}. In this section, we construct the ideal $\mathfrak{h} \trianglelefteq \mathfrak{g}$.  Let $B\subset G$, $B_i\subset G_i $, $C_0 $, and $\mathfrak{g}_i$ be as in the previous section. From now on, the metric that we use on  $\mathfrak{g}$ is the one induced by the norm $\vert \cdot \vert _B$, which by construction comes from an inner product.

By Theorem \ref{thm:nsz}, we have 
\[
        \dim (G) \geq \limsup _{i \to \infty} \dim (G_i)   ,  
\]
so after passing to a subsequence, we may assume $\dim (G_i)$ does not depend on $i$.  Let $\Vert \cdot \Vert _i  :   \mathfrak{g}_i  \to \mathbb{R}$ be the norms given by  Lemma \ref{lem:norm-approximates}. Since $\dim (G_i)$ does not depend on $i$,  by Theorem \ref{thm:john}  we can assume, after possibly updating $C_0$, that the norms $\Vert \cdot \Vert _i$ are given by inner products. 

We now bound the structure coefficients of the Lie algebras $\mathfrak{g}_i$,  enabling the use of Lemmas \ref{lem:bch-quant} and \ref{lem:lie-ift}.

\begin{Lem}
    For $i$ large enough, we have 
\begin{equation}\label{eq:enabler}
     \Vert [ \xx , \yy ] \Vert _i \leq      4 C_0^2 \Vert \xx \Vert _{i} \Vert \yy \Vert_{i}  
\end{equation}
for all $\xx, \yy \in \mathfrak{g}_i$. 
\end{Lem}

\begin{proof}
    We directly compute 
\begin{align*}
    \Vert [ \xx , \yy ] \Vert _i  & =  \lim_{t \to 0}   \frac{1}{t^2} \Vert \log ( [  \exp (t\xx) , \exp (t \yy) ] ) \Vert _ i    \\
    & \leq 2 C_0 \lim_{t \to 0} \frac{1}{t^2} \vert \log ( [ \exp (t\xx) , \exp (t \yy) ])  \vert _{B_i} \\
    & \leq 4 C_0 \lim_{t \to 0} \frac{1}{t^2} \Vert [ \exp (t\xx) , \exp (t \yy)]  \Vert _{B_i}\\
    & \leq 4 C_0 ^2 \lim_{t \to 0} \frac{1}{t^2} \Vert   \exp (t\xx)  \Vert _{B_i} \Vert  \exp (t \yy)  \Vert _{B_i} \\
    & \leq  4 C_0^2 \vert \xx \vert _{B_i} \vert \yy \vert_{B_i} \\
    & \leq 4 C_0^2 \Vert \xx \Vert _{i} \Vert \yy \Vert_{i} , 
\end{align*}
where the second line follows from  \eqref{eq:norm-approximates}, the third one from Lemma \ref{lem:norms-are-equal-if-small}, the fourth one from Theorem \ref{thm:gleason}, the fifth one from \eqref{eq:escape-less-than-norm}, and the sixth again from \eqref{eq:norm-approximates}.     
\end{proof}

Let 
\[     \mathfrak{z}_i  : = \{  \xx \in  \mathfrak{g}_i \, \vert \,  \Vert \xx \Vert _i \leq  1   \}     \]
and 
\[    Z_i : = \exp (\mathfrak{z}_i) \subset G_i .   \]
Since $ \vert \cdot \vert _{B_i} \leq  \Vert \cdot \Vert _i$, we have $ Z_i \subset \overline{B}_i $, so after passing to a subsequence,  $\phi_i (Z_i)$ converges  in the Hausdorff topology to a closed subset $Z \subset \overline{B}$.  We will show that  after intersecting $Z $ with a small open identity neighborhood, it is a normal sub-local group of $G$, allowing us to define $\mathfrak{h}$ as its Lie algebra.

Let $M \geq 1$ be an upper bound of the operator norms of the adjoint operators $\Ad _a : \mathfrak{g} \to \mathfrak{g}$ as $a $ ranges over $\overline{B}$. Set
\[ r_0  : = \min \{  \tfrac{1}{256  C_0 ^3} , \tfrac{1}{4C_0 (M+1)} \}   ,\] 
 define 
\[     U   : = \exp ( \{    \vv   \in \mathfrak{g} \, \vert \, \vert  \vv  \vert _ B < r_0 \} )  ,       \]
and let $U_i \subset B_i$ be a sequence of open symmetric sets approximating $U$. Finally, set
\begin{align*}
       H_i  & : = Z_i \cap U_i  ,  \\
       H & : =  Z \cap U .
\end{align*}
Now we show that these sets are normal sub-local groups. 
\begin{Lem}\label{lem:h-i-normal}
    For $i$ large enough, $H_i $ is a normal sub-local group of $G_i$ with normalizing neighborhood $U_i$. 
\end{Lem}
\begin{Lem}\label{lem:h-normal}
    $H $ is a normal sub-local group of $G$ with normalizing neighborhood $U$. 
\end{Lem}

To prove Lemmas \ref{lem:h-i-normal} and \ref{lem:h-normal}, we require a  couple of preliminary results.

\begin{Lem}\label{lem:useful}
    Let $z _i \in Z_i$ and $\ww _i \in \mathfrak{z}_i$ be such that $\vert \ww _i \vert _{B_i} \leq 1$ and $\exp (\ww _i ) = z_i$ for each $i$. 
    \begin{itemize}
        \item If $\vert \ww _i \vert _{B_i}  < \tfrac{r_0}{2} $ for $i$ large enough, then $z _i \in H_i$ for $i$ large enough. 
        \item If $z_i \in H_i$ for $i$ large enough, then $\vert \ww _i \vert _{B_i} < 2 r_0 $ for $i$ large enough.  
    \end{itemize}
    In particular, 
    \begin{equation}\label{eq:h-i-norm-compatible}
        \exp ( \{ \vv \in \mathfrak{z}_i \, \vert \, \vert \vv \vert _{B_i} < \tfrac{r_0}{2} \} ) \subset H_i \subset \exp ( \{ \vv \in \mathfrak{z}_i \, \vert \, \vert \vv \vert _{B_i} < 2 r_0 \} )
    \end{equation}
    for $i$ large enough. 
\end{Lem}

\begin{proof}
Assuming the first claim fails, after passing to a subsequence, we have $\vert \ww _i \vert _{B_i} < \tfrac{r_0}{2}$ but $z_i \notin U_i$ for all $i$. After further passing to a subsequence, we have $\phi_i (z_i) \to z $ for some $ z \in Z $. By Lemma \ref{lem:lie-norm-continuity}, $z = \exp(\ww )$ for some $\ww \in \mathfrak{g}$ with $\vert \ww \vert _B \leq \tfrac{r_0}{2}$. Hence  $z \in U$. Since the sets $U_i$ approximate $U$, this implies  $z_i \in U_i$ for $i$ large enough; a contradiction.

    Assuming the second claim fails, after passing to a subsequence, we have $z_i \in H_i$ but  $\vert \ww_i \vert _{B_i} \geq 2 r_0$ for all $i$. After further passing to a subsequence, we have $\phi_i (z_i) \to z $ for some $z \in \overline{U}$. By Lemma \ref{lem:lie-norm-continuity}, $z = \exp(\ww )$ for some $\ww \in \mathfrak{g}$ with $\vert \ww \vert _B \geq 2r_0$. This implies $z \notin \overline{U}$; a contradiction.  
\end{proof}

\begin{Lem}\label{lem:h-i-to-h}
    The sequence $\phi_i  (H_i)$ converges in the Hausdorff topology to $H$. 
\end{Lem}

\begin{proof}
For each $h \in H = Z \cap U$, there are $z _i \in Z_i $ with $\phi_i (z_i) \to h$. Since the sets $U_i$ approximate $U$, we have $z_i \in U_i$ for $i$ large enough. This implies $z_i \in H_i$ for $i$ large enough. On the other hand, for any sequence $h_i \in H_i  = Z_i \cap U_i $, any accumulation point of the sequence $\phi_i (h_i)$ belongs to $\overline{Z} \cap \overline{U} = \overline{Z \cap U} = \overline{H}$. 
\end{proof}

\begin{proof}[Proof of Lemma \ref{lem:h-i-normal}]
By \eqref{eq:h-i-norm-compatible} and  \eqref{eq:norm-approximates}, for $i$ large enough we have
    \[ 
           H_i  \subset \exp (   \{  \vv \in \mathfrak{z}_i \, \vert \, \Vert \vv \Vert _i \leq 4 C_0 r_0   \}  )   .   
    \]
    Recall that $    4C_0 r_0 \leq  \frac{1}{64 C _0^2}  ,  $   so by \eqref{eq:enabler} and  Lemma \ref{lem:bch-quant} with $C= 4 C_0^2$, we have  
    \begin{equation}\label{eq:h-i-2-in-z-i}
     H_i  ^2 \subset Z_i .    
    \end{equation}
    This shows that $H_i \subset G_i$ is a sub-local group with associated neighborhood $U_i$ for $i$ large enough.

    To prove that $U_i$ is  also a normalizing neighborhood, consider sequences $a_i \in U_i$ and $h _i \in H_i$.   We claim that 
    \begin{equation}\label{eq:conjugate-in-z-i}
        a_i h_i a_i ^{-1}  \in Z_i
    \end{equation}
    for $i$ large enough. 

     By  \eqref{eq:h-i-norm-compatible}, for $i$ large enough there are $\ww _i \in \mathfrak{g}_i$ with $\exp (\ww _i) = h_i$ and  
\begin{equation}\label{eq:w-i-norm}
          \vert \ww _i \vert _{B_i}< 2 r_0  .     
\end{equation}    
    Notice that $a_ih_i a^{-1} = \exp ( \Ad _{a_i} \ww _i ) $, so if \eqref{eq:conjugate-in-z-i} fails,  after passing to a subsequence, we have
    \[    \Vert \Ad _{a_i} \ww _i \Vert_i \geq 1     \]
    for all $i$.  Then define
    \begin{equation}\label{eq:lambda-i-bound}
             \lambda_i : = \vert \Ad _{a_i} \ww _i \vert _{B_i} \geq \tfrac{1}{2C_0} \Vert \Ad _{a_i} \ww _i \Vert _i \geq \tfrac{1}{2C_0}     
    \end{equation}
    and 
    \[   \vv _i : = \Ad _{a_i} \ww _i / \lambda_i    .       \]
    After passing to a subsequence, we have $\phi_i (a_i) \xrightarrow[i \to \infty]{} a$ for some $a\in \overline{U}$ and 
    \begin{align}
        \phi_i (\exp (\vv _i )) &\xrightarrow[i \to \infty ]{} \exp (\vv ) , \label{eq:vvv}   \\
            \phi_i ( \exp ( \ww_i / \lambda_i )   ) &\xrightarrow[i \to \infty ]{} \exp (  \ww  )   , \label{eq:ww-prime}     
    \end{align}
    for some $\vv, \ww  \in \mathfrak{g}$ with
    \[ \hspace{0.5cm} \vert \vv \vert _B = 1, \hspace{2cm} \vert \ww  \vert _B \leq 4 C_0 r_0 ,\]
    where the last inequality follows from \eqref{eq:w-i-norm} and \eqref{eq:lambda-i-bound}.    
    
    By the definition of $\vv_i$, we also have
    \begin{equation}\label{eq:adjoint-i}
            \exp ( \vv _i ) =   a_i \exp (\ww _i / \lambda_ i ) a_i ^{-1}   ,      
    \end{equation}
    so combining \eqref{eq:vvv}, \eqref{eq:ww-prime}, and \eqref{eq:adjoint-i} we get
    \begin{align*}
     \exp (\vv) & = \lim _{i \to \infty } \phi_i (\exp (\vv _ i )) \\ 
     & = \lim_{i \to \infty } \phi_i (   a_i \exp (\ww _i / \lambda _i ) a_i^{-1}   ) \\
     & = a \exp (\ww  ) a^{-1} \\
     & = \exp ( \Ad _a (\ww ) ).
    \end{align*}
    This implies
    \[    \Vert \Ad _a \Vert \geq \frac{ \vert \vv \vert  _B }{ \vert \ww  \vert _B  } \geq \frac{1}{4C_0 r_0 } \geq M +1 ,                     \]
    contradicting our choice of $M$. This proves \eqref{eq:conjugate-in-z-i}.    Since $a_i \in U_i$ and $h_i \in H_i$ were arbitrary, \eqref{eq:conjugate-in-z-i} implies that $U_i$ is a normalizing neighborhood of $H_i$ for $i$ large enough. 
\end{proof}

\begin{proof}[Proof of Lemma \ref{lem:h-normal}]
    Fix $h_1, h_2 \in H$. By Lemma \ref{lem:h-i-to-h}, there are $h_{i,1}, h_{i,2} \in H_i$ with 
    \[ \phi_i (h_{i,j}) \xrightarrow[i \to \infty]{} h_j \] 
    for $j \in \{ 1, 2\}$. By \eqref{eq:h-i-2-in-z-i} we have
    \[   h_{i,1} h_{i,2} \in Z_i         \]
    for $i$ large enough. This implies $h_1h_2 \in Z$. Since $h_1, h_2 \in H$ were arbitrary, this shows that $H$ is a  sub-local group with associated neighborhood $U$. 
    
    To prove that $U$ is also a normalizing neighborhood, pick $a \in U$ and $h \in H$. By Lemma \ref{lem:h-i-to-h}, there are  $a_i \in U_i $ and $h _i \in Z_i $ with $\phi_i (a_i) \to a$ and $\phi_i (h_i ) \to h$. From   \eqref{eq:conjugate-in-z-i},  we deduce that $a h a^{-1} \in Z$. Since $a \in U$ and $h\in H$ were arbitrary, this shows that $U$ is a normalizing neighborhood of $H$. 
\end{proof}

We can then define 
\begin{equation}\label{eq:h-definition}
   \mathfrak{h} : = \bigcup_{\varepsilon > 0 } \{ \vv \in \mathfrak{g} \, \vert \, \exp (t \vv) \in H  \text{ for all } t \in ( - \varepsilon , \varepsilon)   \} .        
\end{equation}
Since $H$ is a normal sub-local group of $G$, Proposition \ref{pro:lie-local} implies that  $\mathfrak{h}$ is an ideal of $\mathfrak{g}$ and $H \subset G$ is a smooth submanifold of dimension $\dim (\mathfrak{h})$.

\section{Properties of limit Lie algebra}\label{sec:properties-of-h}

In this section we finish the proof of Theorem \ref{thm:main-nss}.   Let $B\subset G$, $B_i\subset G_i $, $C_0$, $\mathfrak{g}_i$, $\Vert \cdot \Vert _i$, $\mathfrak{z}_i$, $Z_i$, $Z$, $r_0$, $U_i$, $U$, $H_i$, $H$,  and  $\mathfrak{h}$ be as in the previous section.  We aim to show that   $\dim (\mathfrak{h}) = \dim (G_i)$ and that $\mathfrak{g}/\mathfrak{h}$ is nilpotent.  We begin with one of the desired inequalities.

\begin{Lem}\label{lem:first-half}
    $\dim (\mathfrak{h}) \leq  \dim (G_i)$. 
\end{Lem}

\begin{proof}
Set $k : = \dim (G_i)$.  For each $i$, let $\{ \vv_{i,1}, \ldots, \vv_{i,k}  \} \subset \mathfrak{g}_i$ be an $\Vert \cdot \Vert _i$-orthonormal basis.  By Proposition \ref{pro:convergence-of-ops}, after passing to a subsequence, for each $j$, the  one-parameter subgroups $\gamma_{i,j} : \mathbb{R} \to G_i$ given by $\gamma_{i,j} (t) : = \exp (t \vv _{i,j})$ converge as $i \to \infty$ to a one-parameter subgroup $\gamma _j : \mathbb{R} \to G$ of the form $\gamma _j (t)  = \exp (t \vv _j ) $ with $\vert \vv _j \vert _B  \in [ \frac{1}{2C_0} ,  1  ] $.  

Define maps $\Theta _i : \mathbb{R} ^k \to G_i$ and $\Theta : \mathbb{R} ^k \to G $ by
\begin{align*}
    \Theta _i (t_1, \ldots , t_k) & : = \exp (t_1 \vv_{i,1})\cdots \exp (t_k \vv_{i,k} ) , \\
     \Theta (t_1, \ldots , t_k) & : = \exp (t_1 \vv_1)\cdots \exp (t_k \vv_k ) . 
\end{align*}
Let  $\delta > 0 $ be given by Lemma \ref{lem:lie-ift} with $C  = 4 C_0^2 $ and  $\varepsilon  = 1$.  Seeking a contradiction, assume $\dim (\mathfrak{h}) > k$. Then by Sard's Theorem, there is 
\[ h_0 \in H \cap \exp (B_{\delta / 2 C_0}^{\mathfrak{g}}(0)) \, \backslash \, \Theta (\mathbb{R}^k). \] 
Pick $\ww _i \in \mathfrak{z}_i$ with $\phi_i (\exp (\ww_i)) \to h_0$. By Lemma \ref{lem:lie-norm-continuity}, for $i$ large enough we have
\[   \Vert \ww _i \Vert _i \leq 2 C_0 \vert \ww _i \vert _{B_i} < \delta .          \]
Then, by our choice of $\delta $,  there are $t_{i,1} , \ldots , t_{i,k } \in ( -1  ,  1  ) $ with 
\[   \Theta_i ( t_{i,1} , \ldots , t_{i,k} ) =     \exp (\ww _i ) .     \]
After passing to a subsequence once more, we can assume $t_{i,j} \xrightarrow[i \to \infty ]{} t_j \in [- 1   ,  1  ]$ for each $j \in \{ 1, \ldots , k \}$. Then 
\begin{align*}
    \Theta (t_1, \ldots , t_k) & = \exp (t_1 \vv_1)\cdots \exp (t_k \vv_k ) \\
    & = \lim_{i \to \infty} \phi_i ( \exp (t_{i,1} \vv_{i,1} ) \cdots \exp (t_{i,k} \vv_{i,k} ) ) \\
    & = \lim_{i \to \infty} \phi_i (  \Theta _i  (t_{i,1} , \ldots , t_{i,k})  ) \\
    & = \lim_{i \to \infty } \phi_i ( \exp (\ww _i ) ) \\
    & = h_0 ,
\end{align*}
contradicting our choice of $h_0$.  \end{proof}

Now we aim to prove $\dim (\mathfrak{h}) \geq \dim (G_i)$. In view of Theorem \ref{thm:nsz-local}, we seek to construct local good approximations $\phi_i ' : H_i \to H$.

Equip $G$ with a metric $d : G \times G \to \mathbb{R}$ compatible with its topology.  By Lemma \ref{lem:h-i-to-h}, there is a sequence $\delta_i \to 0$ such that the Hausdorff distance between $\phi_i (H_i) $ and $H $ is less than $\delta_i$. Then define 
\[   \phi' _i :  H_i \to H    \] 
to be any map with 
\begin{equation}\label{eq:phi-close-to-original}
d (  \phi'_i (x) , \phi_i (x)  ) < \delta _i    
\end{equation}
for all $x \in H_i$.

Let $\varepsilon \in (0,r_0 ) $ and define $A_i \subset G_i$ and $A \subset H$ as
\begin{align*}    
    A_i & : = \exp (\{  \vv \in \mathfrak{g}_i \, \vert \, \vert \vv \vert _{B_i } < \varepsilon \} ) ,\\
     A &  : = \exp ( \{ \vv \in \mathfrak{h} \, \vert \, \vert \vv \vert _{B } < \varepsilon \} ) .       
\end{align*}
By  \eqref{eq:h-i-norm-compatible} and Lemma \ref{lem:bch-quant}, we can choose  $\varepsilon$  so that  $\overline{A}^ {200} \subset H $ and $\overline{A}^{200}_i \subset H_i $ for all $i$. 

\begin{Lem}\label{lem:phi-prime-local-good}
The maps $\phi_i ' : H _i \to H  $ are local good approximations with regular neighborhoods $A_i\subset H_i$ and $A\subset H$.  Moreover, the sequence $H_i$ has the NSS property.
\end{Lem}

\begin{proof}
Using \eqref{eq:phi-close-to-original}, each property in the definition of local good approximation for $\phi_i ' $ follows from the corresponding property for $\phi_i$ together with Lemma \ref{lem:lie-norm-continuity}.  By \eqref{eq:phi-close-to-original}, any sequence of small subgroups $K_i \leq H_i$ with respect to the pairs $(\phi_i ' , A_i)$ is also small with respect to the pairs $(\phi_i, B_i)$, so the sequence $H_i$ has the NSS property. 
\end{proof}

We are now  ready to prove the other desired inequality and finish the proof of Theorem \ref{thm:main-nss}.

\begin{Lem}\label{lem:second-half}
    $\dim (\mathfrak{h}) \geq \dim (G_i)$. 
\end{Lem}

\begin{proof}
Since $H_i \subset G_i$ is open for each $i$, we have $\dim (G_i) = \dim (H_i)$. Then  the result follows from Lemma \ref{lem:phi-prime-local-good} and Theorem \ref{thm:nsz-local}.
\end{proof}

\begin{Lem}\label{lem:third-half}
    The Lie algebra $\mathfrak{g}/\mathfrak{h}$ is nilpotent. 
\end{Lem}

\begin{proof}    
By Lemma \ref{lem:approximate} and Proposition \ref{pro:lga-quotient},  one can construct open symmetric sets  $W_i \subset B_i$ and $W \subset B $  and local good approximations 
\[ \psi _i :  (G_i/H_i)_{W_i}    \to (G/H)_W .                \]
Since $H_i \subset G_i $ is open for each $i$, the local groups $(G_i/ H_i) _{W_i}$ are discrete. Then Theorem \ref{thm:bgt-local} implies that the Lie algebra of $(G/H)_W$ is nilpotent. By Proposition \ref{pro:lie-local}, this nilpotent Lie algebra is isomorphic to   $\mathfrak{g}/\mathfrak{h}$.  
\end{proof}

\begin{proof}[Proof of Theorem \ref{thm:main-nss}]
Using the notation of Sections \ref{sec:ops-convergence} and \ref{sec:h-construction}, define $\mathfrak{h}$ as in \eqref{eq:h-definition}. By Lemmas \ref{lem:first-half}, \ref{lem:second-half}, and \ref{lem:third-half},  $\mathfrak{h}$ satisfies the desired properties. 
\end{proof}

\begin{Rem}\label{rem:main-local}
    Theorem \ref{thm:main-nss} also holds for local good approximations with the NSS property. The same proof carries over with minimal modifications. 
\end{Rem}

\begin{proof}[Proof of Theorem \ref{thm:main}]
By Theorem \ref{thm:egh-to-ga}, the maps $\phi _i : G_i \to G $ demonstrating the convergence \eqref{eq:egh} are good approximations. By Theorem \ref{thm:main-good}, there is an ideal $\mathfrak{h} \trianglelefteq \mathfrak{g}$ with $\mathfrak{g}/\mathfrak{h}$ nilpotent and satisfying \eqref{eq:h-ineq}. 

Assume $\vol (B_1 (p_i)) \geq v $ for some $v>0$. By Theorem \ref{thm:egh-to-ga}, the sequence $G_i$ has the NSS property, so by Theorem \ref{thm:main-good}, $\mathfrak{h} $ can be taken so that equality in \eqref{eq:h-ineq} holds.    \end{proof}

\section{Proofs of corollaries}\label{sec:proofs-of-corollaries}

Before we prove Corollaries \ref{cor:no-nilpotent} and \ref{cor:no-nilpotent-2}, we need a couple of preliminary results. The first one gives a criterion on when a connected Lie group is topologically perfect.

\begin{Pro}\cite{abbgnrs}\label{pro:seven-samurai}. 
    Let $G$ be a connected Lie group. Then $G$ is not topologically perfect if and only if for any  compact set $K \subset G$ and any open set $U \subset G$, there is a closed proper subgroup $H \leq G$ with $K \subset HU$. 
\end{Pro}

The following lemma allows one to replace the groups in a sequence of good approximations by groups given by Proposition \ref{pro:seven-samurai}. 

\begin{Lem}\label{lem:restrictions-of-dense}
Let $G_i$ be a sequence of locally compact Hausdorff groups, $G$ a metrizable locally compact group, and open symmetric subsets $H_i \subset G_i$ and $H \subset G$.  Let $\phi_i : H_i \to H$ be a sequence of local good approximations with regular neighborhoods $A_i \subset H_i$ and $A \subset H$. Assume that for each $i$, there is a sequence of subgroups $H_{i,j} \subset G_i$ such that for each  identity neighborhood $V \subset G_i$, one has $A_i \subset H_{i,j} V$ for $j$ large enough. If for each $i$, we choose $j (i) $ sufficiently large, then the  restrictions 
 \begin{equation}\label{eq:restrictions-of-dense}
          \phi_i : H_i \cap H_{i,j(i)} \to H          
    \end{equation}
    are local good approximations with regular neighborhoods $A_i \cap H_{i,j(i)} \subset H_{i,j (i)}$ and $A \subset H$. 
\end{Lem}

\begin{proof}
Since $G$ is metrizable, using \eqref{item:lga-5} and a diagonal argument, we can find a sequence of open symmetric subsets $V_i \subset A_i$ such that $\phi_i (x_i )  \to e_H$ for any sequence $x_i \in V_i$. For each $i$, choose $j(i)$ large enough so that $A _i \subset  H_{i,j(i)} V_i$. 

To prove that the maps \eqref{eq:restrictions-of-dense} satisfy \eqref{item:lga-1}, fix an open non-empty subset $U \subset A $. Note that we can assume $\overline{U} \subset A$.  Using the fact that the maps $\phi _i : H_i \to H$ satisfy \eqref{item:lga-1}, we can find $g_i \in A_i$ such that $\phi_i (g_i) \to g \in U$. By our choice of $j(i)$, there is a sequence $h_i \in H_{i,j(i)}$ such that $h_i^{-1} g_i \in V_i $ for all $i$. Therefore, 
\[     \phi_i (h_i) = [\phi_i (h_i) \phi_i ( h_i ^{-1}g_i  ) \phi_i (g_i ) ^{-1}    ]  \ast   [ \phi_i ( g_i )  ]  \ast [  \phi_i ( h_i ^{-1} g_i)^{-1}  ]      .   \]
By \eqref{item:lga-3}, the first and third terms converge to $e_H$, while the middle one converges to $g$. Hence  $\phi_i (h_i) \to g$ and $\phi_i (h_i) \in U $ for $i$ large enough. Since $\overline{U } \subset A$ and the maps $\phi_i : H_i \to H$ satisfy \eqref{item:lga-4}, we also have $h_i \in A_i \cap H_{i,j (i)}$ for $i$ large enough. This proves that the maps \eqref{eq:restrictions-of-dense} satisfy \eqref{item:lga-1}.  

The remaining properties (\ref{item:lga-2}-\ref{item:lga-5}) for the maps \eqref{eq:restrictions-of-dense} follow immediately from the corresponding ones for the maps $\phi_i : H_i \to H$. 
\end{proof}

Now we are ready to prove Corollaries \ref{cor:no-nilpotent} and \ref{cor:no-nilpotent-2}.

\begin{proof}[Proof of Corollary \ref{cor:no-nilpotent}]
 By Theorem \ref{thm:egh-to-ga}, the maps $\phi_ i : G_i \to G$  demonstrating the  convergence are good approximations with regular neighborhoods chosen so that the sequence $G_i$ has the NSS property.

Let $\mathfrak{g}$ be the Lie algebra of $G$ and let $\mathfrak{h} \leq \mathfrak{g}$ be given by Theorem \ref{thm:main-nss}.   Since $\mathfrak{g}$  is perfect, it does not admit non-trivial nilpotent quotients, so $\mathfrak{h} = \mathfrak{g}$. This implies 
\[   \limsup_{i \to \infty}  \dim (G_i )  = \dim (\mathfrak{h}) = \dim (G) .             \]
Since the same conclusion holds for any subsequence, we have $\dim (G_i) = \dim (G)$ for $i$ large enough. 

Let $H_i \subset G_i$ and $H \subset G $ be the sub-local groups constructed in the proof of Theorem \ref{thm:main-nss} (see Section \ref{sec:h-construction}) and consider the local good approximations $\phi_i ' : H_i \to H$ with regular neighborhoods $A_i \subset H_i$ and $A \subset H$ constructed in Section \ref{sec:properties-of-h} (see Lemma \ref{lem:phi-prime-local-good}).  

Note that for each $i$,  the subset $H_i \subset G_i $ contains an identity neighborhood and is contained in the image of the exponential map, so the subgroup  $G_i ' : = \langle H_i \rangle \leq G_i $ coincides with the identity component of $G_i$.  Arguing by contradiction, after passing to a subsequence, we can assume  $G_i' $ is not perfect for each $i$. By Proposition \ref{pro:seven-samurai}, for each $i$ there is a sequence of closed proper subgroups $H_{i,j}  \leq G_i'  $ such that for any identity neighborhood $V \subset  G_i  $, one has $A_i \subset H_{i,j} V$ for $j$ large enough.

By Lemma \ref{lem:restrictions-of-dense}, if for each $i$ we choose $j(i)$ sufficiently large, the restrictions 
\[      \phi _i '  : H_i \cap  H_{i,j(i)}   \to    H  \]
 are local good approximations.  Note that any sequence of small subgroups $K_i \leq H_i  \cap H_{i,j(i)}$ consists of small subgroups $K _i  \leq H_i$, so the sequence $H_i \cap H_{i,j(i)}$ also has the NSS property. 

 Repeating the proof of Theorem \ref{thm:main-nss} (see Remark \ref{rem:main-local}), we obtain an ideal $\mathfrak{h}' \trianglelefteq \mathfrak{h}$ with $\mathfrak{h} / \mathfrak{h} ' $ nilpotent and 
 \[    \dim (\mathfrak{h}' ) = \limsup_{i \to \infty} \dim (H_{i} \cap H_{i,j(i)}) <  \limsup _{i \to \infty} \dim (H_i) = \dim (\mathfrak{h}) .                        \]
This contradicts the fact that $\mathfrak{h} $ is perfect and finishes the proof. 
\end{proof}

\begin{proof}[Proof of Corollary \ref{cor:no-nilpotent-2}]
    By \cite[Theorem C]{nsz}, there is a sequence of good approximations $\phi_i : G_i \to G $ with regular neighborhoods $A_i \subset G_i $ and $A \subset G$ with the property that any sequence of small subgroups $H_i \leq G_i$ with respect to the pairs $(\phi_i , A_i)$ is also small as in Remark \ref{rem:small}.  

    Let $O_i : = \langle A_i \rangle$, $H_i \trianglelefteq O_i$ the sequence of small subgroups given by Theorem  \ref{thm:largest-small}, and $G_i ' : = O_i / H_i $. By Lemma \ref{lem:small-quotient}, there is a sequence of good approximations $\phi_i ' : G_i ' \to G$ so that the sequence $G_i'$ has the NSS property (see the proof of Theorem \ref{thm:main-good} in Section \ref{sec:to-nss}).   Arguing as in the proof of Corollary \ref{cor:no-nilpotent}, the identity component of $G_i'$ is topologically perfect and $\dim (G_i' ) = \dim (G)$ for $i$ large enough. 
\end{proof}

\section*{Acknowledgements}

 The author would like to thank Logan Richard, Christine Escher,  Jiayin Pan, Jes\'us N\'u\~nez-Zimbr\'on, and Jaime Santos-Rodr\'iguez for helpful discussions. During the preparation of this paper, the author was a Postdoctoral Scholar at Oregon State University.

\printbibliography

\end{document}